\documentclass[12pt,a4paper,oneside,reqno]{amsart} 

\usepackage{amsthm}
\usepackage{amsmath}
\usepackage{amssymb}
\usepackage{amscd}
\usepackage{mathtools}
\usepackage[utf8]{inputenc}
\usepackage[dvipsnames]{xcolor}
\usepackage{typearea}
\usepackage{eufrak}
\usepackage{yfonts}
\usepackage{textcomp}
\usepackage{mathrsfs}

\usepackage{todonotes}
\usepackage{pdfsync}
\usepackage[active]{srcltx}
\usepackage{xypic}
\usepackage{verbatim}
\usepackage{tikz}
\usepackage{stmaryrd} 
\usepackage{subcaption}  
\usepackage[hidelinks]{hyperref} 

\usetikzlibrary{cd, decorations.markings, arrows, matrix, calc}

\DeclareMathOperator{\id}{id}
\DeclareMathOperator{\im}{im}
\DeclareMathOperator{\rank}{rank}
\DeclareMathOperator{\supp}{supp}

\newcommand{\ZZ}{\mathbb{Z}}
\newcommand{\NN}{\mathbb{N}}

\newcommand{\DD}{\mathsf{D}}
\newcommand{\GG}{\mathcal{G}}

\newcommand{\orb}{\operatorname{Orb}}

\newcommand{\G}{\mathcal{G}}
\newcommand{\spann}{\textrm{span}}

\newcommand{\coker}{\operatorname{coker}}

\newtheorem{lemma}{Lemma}[section]
\newtheorem{corollary}[lemma]{Corollary}
\newtheorem{theorem}[lemma]{Theorem}
\newtheorem{introtheorem}{Theorem}

\newtheorem*{theorem*}{Theorem}
\newtheorem{proposition}[lemma]{Proposition}

\newtheorem*{matuiAH*}{Matui's AH~Conjecture} 
\theoremstyle{definition}
\newtheorem{definition}[lemma]{Definition}

\newtheorem{example}[lemma]{Example}

\newtheorem{remark}[lemma]{Remark}

\makeatletter
\def\l@section{\@tocline{1}{0pt}{1pc}{}{}}
\def\l@subsection{\@tocline{2}{0pt}{1pc}{4.6em}{}}
\def\l@subsubsection{\@tocline{3}{0pt}{1pc}{7.6em}{}}
\renewcommand{\tocsection}[3]{%
  \indentlabel{\@ifnotempty{#2}{\makebox[2.3em][l]{%
    \ignorespaces#1 #2.\hfill}}}#3}
\renewcommand{\tocsubsection}[3]{%
  \indentlabel{\@ifnotempty{#2}{\hspace*{2.3em}\makebox[2.3em][l]{%
    \ignorespaces#1 #2.\hfill}}}#3}
\renewcommand{\tocsubsubsection}[3]{%
  \indentlabel{\@ifnotempty{#2}{\hspace*{4.6em}\makebox[3em][l]{%
    \ignorespaces#1 #2.\hfill}}}#3}
\makeatother

\makeatletter
\newcommand\@dotsep{4.5}
\def\@tocline#1#2#3#4#5#6#7{\relax
  \ifnum #1>\c@tocdepth 
  \else
    \par \addpenalty\@secpenalty\addvspace{#2}%
    \begingroup \hyphenpenalty\@M
    \@ifempty{#4}{%
      \@tempdima\csname r@tocindent\number#1\endcsname\relax
    }{%
      \@tempdima#4\relax
    }%
    \parindent\z@ \leftskip#3\relax
    \advance\leftskip\@tempdima\relax
    \rightskip\@pnumwidth plus1em \parfillskip-\@pnumwidth
    #5\leavevmode\hskip-\@tempdima #6\relax
    \leaders\hbox{$\m@th
      \mkern \@dotsep mu\hbox{.}\mkern \@dotsep mu$}\hfill
    \hbox to\@pnumwidth{\@tocpagenum{#7}}\par
    \nobreak
    \endgroup
  \fi}
 \def\l@subsection{\@tocline{2}{0pt}{30pt}{5pc}{}}
\makeatother

\title[Matui's AH~conjecture for Graph Groupoids]{Matui's AH~conjecture for Graph Groupoids}

\author[1]{Petter Nyland}
\author[2]{Eduard Ortega}

\address{Department of Mathematical Sciences, Faculty of Information Technology and Electrical Engineering, NTNU -- Norwegian University of Science and Technology, Trondheim, Norway}
\email{petter.nyland@ntnu.no}
\address{Department of Mathematical Sciences, Faculty of Information Technology and Electrical Engineering, NTNU -- Norwegian University of Science and Technology, Trondheim, Norway}
\email{eduard.ortega@ntnu.no}



\numberwithin{equation}{section} 

\begin{document}

\begin{abstract}
We prove that Matui's AH~conjecture holds for graph groupoids of infinite graphs. This is a conjecture which relates the topological full group of an ample groupoid with the homology of the groupoid. Our main result complements Matui's result in the finite case, which makes the AH~conjecture true for all graph groupoids covered by the assumptions of said conjecture. Furthermore, we observe that for arbitrary graphs, the homology of a graph groupoid coincides with the $K$-theory of its groupoid \mbox{$C^*$-algebra}.
\end{abstract} 



\maketitle

 \tableofcontents 	






\section{Introduction}
\subsection{Background}

Building on the discoveries in the series of papers~\cite{MatRem}, \cite{MatHom} and~\cite{MatTFG}  Hiroki Matui stated two conjectures concerning effective minimal étale groupoids over Cantor spaces in~\cite{MatProd}.
The \emph{HK~conjecture} predicts that the \mbox{$K$-theory} of a reduced groupoid \mbox{$C^*$-algebra} is determined by the groupoid's homology as follows:
\[ K_0\left(C^*_r(\mathcal{G})\right) \cong \bigoplus_{n=0}^\infty H_{2n}(\mathcal{G}) \quad \text{and} \quad K_1\left(C^*_r(\mathcal{G})\right) \cong \bigoplus_{n=0}^\infty H_{2n+1}(\mathcal{G}).  \]
 The \emph{AH~conjecture} predicts that the abelianization of the topological full group of a groupoid together with its first two homology groups fit together in an exact sequence as follows: \[\begin{tikzcd}
H_0(\mathcal{G}) \otimes \mathbb{Z}_2 \arrow{r}{j} & \llbracket \mathcal{G} \rrbracket_{\text{ab}} \arrow{r}{I} & H_1(\mathcal{G}) \arrow{r} & 0.
\end{tikzcd}  \]
In several cases (including graph groupoids) the $K$-groups actually coincide with the two first homology groups, which means that the AH~conjecture in these cases relates the $K$-theory of the groupoid \mbox{$C^*$-algebra} with the topological full group. 

Topological full groups associated to dynamical systems (and more generally to étale groupoids) are perhaps best known for being complete invariants for continuous orbit equivalence (and groupoid isomorphism). And also for diagonal preserving isomorphism of the associated \mbox{$C^*$-algebrs}. Roughly speaking, the topological full group consists of all homeomorphisms which preserve the orbits of the dynamical system in a continuous manner. Consult~\cite{GPS99}, \cite{Med11}, \cite{Matsumoto}, \cite{MatTFG}, \cite{NO} and~\cite{CGW} for some of these rigidity results. Topological full groups also provide means of constructing new groups with interesting properties, most notably by providing the first examples of finitely generated simple groups that are amenable (and infinite)~\cite{JM}.

In the works of Matui mentioned above, both conjectures were verified for key classes of groupoids, such as AF-groupoids, transformation groupoids of minimal $\ZZ$-actions and groupoids associated to shifts of finite type (\emph{SFT-groupoids}). Subsequently, other authors have expanded upon this. The HK~conjecture has been shown to hold for Katsura--Exel--Pardo groupoids~\cite{Ort}, Deaconu--Renault groupoids of rank~$1$ and $2$~\cite{FKPS} and groupoids of unstable equivalence relations on one-dimensional solenoids~\cite{Yi20}.

Alas, the HK~conjecture is now known to be false in general. It fails to hold for transformation groupoids associated to odometers on the infinite dihedral group, as demonstrated in~\cite{Scarp}. Nevertheless, it is still interesting to investigate for which groupoids the conclusion of the HK~conjecture holds. We will say that a groupoid has the \emph{HK~property} when this is the case. In spite of them providing counterexamples to the HK~conjecture, the AH~conjecture was shown, also in~\cite{Scarp}, to hold for transformation groupoids arising from odometers. Hence the AH~conjecture remains open. A notable difference between the two conjectures is that in the AH~conjecture the maps involved are specified, whereas in the HK~conjecture it is only predicted that some isomorphisms exist.



\subsection{Our results}

The purpose of this paper is to investigate the AH~conjecture for the class of graph groupoids. As the SFT-groupoids prominently studied by Matui can be realized as graph groupoids of finite  graphs, the novelty lies in dealing with infinite (directed) graphs. In particular with the presence of infinite emitters, i.e.\ vertices that emit infinitely many edges. 

Our main motivating example has been the graph $E_\infty$ which has one vertex and infinitely many loops. The graph groupoid $\mathcal{G}_{E_\infty}$ is the canonical groupoid model for the (infinitely generated) Cuntz algebra~$\mathcal{O}_\infty$. This was a natural example to explore as~$E_\infty$ is the simplest possible graph having an infinite emitter. On the other hand, its graph~\mbox{$C^*$-algebra}~$\mathcal{O}_\infty$ has played---and continues to play---an important role in the theory of  \mbox{$C^*$-algebras}. Seeing as the topological full groups of the canonical graph groupoid models of the other Cuntz algebras $\mathcal{O}_n$ are isomorphic to the highly interesting \mbox{Higman--Thompson} groups~$V_{n,1}$, we believe it worthwhile to also investigate the topological full group $\llbracket \mathcal{G}_{E_\infty} \rrbracket$. 

One of the assumptions in the AH~conjecture is that the unit space of the groupoid is compact, and this translates into the underlying graph having finitely many vertices. We were indeed able to show that the AH~conjecture holds for these graph groupoids as well, so that our main result is the following. 

\begin{introtheorem}[see Corollary~\ref{cor:AH}]\label{thm:AHintro}
Let $E$ be a strongly connected graph with finitely many vertices which is not a cycle graph. Then the AH~conjecture holds for the graph groupoid~$\mathcal{G}_E$.
\end{introtheorem}

Let us remark that Corollary~\ref{cor:AH} applies to a slightly more general family of graphs than in the preceding theorem, as well as to all restrictions of these graph groupoids. The conclusion is that the AH~conjecture holds for all graph groupoids covered by the assumptions in said conjecture. Additionally, it holds for any groupoid which is Kakutani equivalent to such a graph groupoid.


It should be mentioned that Matui in~\cite{MatTFG} not only proved that the AH~conjecture is true for restrictions of SFT-groupoids, but that these also have the \emph{strong AH~property}. This means that the map $j$ is injective, so that one has a short exact sequence. This was done by constructing a suitable finite presentation of the topological full group. We investigate this subject in Section~\ref{sec:examples}, but we find that when the graph has an infinite emitter, then the topological full group is not even finitely generated. 

We also observe that all graph groupoids have the HK~property. The following theorem is but a small extension of already existing results (see the paragraph following Theorem~\ref{thm:graphhom}). 
\begin{introtheorem}[see Theorem~\ref{thm:graphhom}]\label{thm:graphhomintro}
Let $E$ be any graph. Then 
\begin{align*}
H_0(\mathcal{G}_E) & \cong K_0(C^*(E)), \\
H_1(\mathcal{G}_E) & \cong K_1(C^*(E)), \\
H_n(\mathcal{G}_E) &= 0, \quad n \geq 2.
\end{align*}
\end{introtheorem}
Here $C^*(E)$ denotes the graph \mbox{$C^*$-algebra} of $E$, which is canonically isomorphic to the groupoid \mbox{$C^*$-algebra} $C^*_r(\mathcal{G}_E)$. Since the $K$-groups of a graph \mbox{$C^*$-algebra} are relatively easy to compute, Theorem~\ref{thm:graphhomintro} allows us to give a partial description of the abelianization of the topological full group $\llbracket \mathcal{G}_E \rrbracket_{\text{ab}}$ via the AH~conjecture.

Our proof of the AH~conjecture for graph groupoids of infinite graphs will in broad strokes follow a similar strategy as Matui's proof for finite graphs from~\cite{MatTFG}. However, we emphasize that there are several major differences which make this a nontrivial generalization. There are steps and techniques in Matui's proof that no longer work---or even make sense---in the infinite setting. A couple of significant differences are described below.

If $E$ is a graph with infinite emitters (or sinks), then the unit space of its graph groupoid is no longer full in the associated skew product (compare~\cite[Lemma~6.1]{FKPS} and Remark~\ref{rem:fullness}). This means that we cannot deduce that the kernel of the canonical graph cocycle is Kakutani equivalent to the skew product, and in turn we cannot identify their homologies as is done in Matui's proof.  


A key component in Matui's proof is the reduction it to \emph{mixing} shifts of finite type. This is equivalent to the adjacency matrix of the associated finite graph  being \emph{primitive}. In this case, the kernel of the cocycle is a minimal AF-groupoid admitting a unique invariant probability measure arising from the Perron eigenvalue of the adjacency matrix. This measure can then be used to compare clopen subsets of the unit space and produce certain bisections connecting them. When passing to the infinite setting we lose all of this. We no longer have a shift of finite type (nor any shift space for that matter) and no Perron--Frobenius theory. Furthermore, the kernel of the cocycle is not minimal anymore. 


We also wish to remark that even though certain parts of the paper are quite similar to parts of~\cite[Section~6]{MatTFG}, such as Section~\ref{sec:indexmap} and the second half of the proof of Theorem~\ref{thm:TR}, we have chosen to keep the exposition mostly self-contained. We have done this in the best interest of the reader. For there are still subtle differences, such as indexes being shifted or reversed, and some steps being done in the opposite order. This is in part due to us having to consider the inverse of a certain map from Matui's proof, see Remarks~\ref{rem:MatuiDelta}  and~\ref{rem:deltadifference}. We supply several remarks along the way which compare our approach to Matui's to signify where they differ.

The work laid down in this paper is not  done with graph groupoids alone in mind. It is our belief that these techniques can also be applied to other groupoids which have an underlying ``graph skeleton'', such as groupoids arising from self-similar actions by groups on graphs, as studied by Nekrashevych~\cite{Nek09} and by Exel and Pardo~\cite{EP17}. The authors plan to explore this avenue in future work. Groupoids associated to $k$-graphs and ultragraphs are also obvious candidates.

\subsection{Summary}
We begin in Section~\ref{sec:etale} by giving the necessary background regarding étale groupoids.
This includes the topological full group, homology and skew products by cocycles. More background is given in Section~\ref{sec:graphs},  regarding graphs and their associated groupoids. The graph groupoid $\mathcal{G}_E$ associated to a graph~$E$ has a canonical~$\ZZ$-valued cocycle denoted~$c_E$. Both the skew product groupoid $\G_E \times_{c_E} \ZZ$ and the kernel subgroupoid~$\mathcal{H}_E \coloneqq \ker(c_E) \subseteq \mathcal{G}_E$ play important roles in the rest of the paper. We show that the graph groupoids of acyclic graphs are AF-groupoids. From this we deduce that both~$\G_E \times_{c_E} \ZZ$ and $\mathcal{H}_E$ are AF-groupoids.

In Section~\ref{sec:AH} we describe the AH~conjecture in more detail. One of the maps appearing in the AH~conjecture is the \emph{index map} $I \colon \llbracket \mathcal{G} \rrbracket \to H_1(\G)$.
We extend its definition to groupoids with non-compact unit space. Then the assumptions in the AH~conjecture for graph groupoids are translated into properties of the underlying graphs. These turn out to be equivalent to the graph \mbox{$C^*$-algebra} being a unital Kirchberg algebra. We also note that all graph groupoids have the HK~property by combining known results in the row-finite case with the concept of desingularization. This yields Theorem~\ref{thm:graphhomintro}. The graph groupoids satisfying the assumptions in the AH~conjecture are shown to be purely infinite. It then follows from a result of Matui (see Remark~\ref{rem:TR}) that the AH~conjecture is equivalent to \emph{Property~TR}. Property~TR means that the kernel of the index map is generated by transpositions.  Hence the rest of the paper, except for the final section, is devoted to establishing Property~TR for these graph groupoids.

Section~\ref{sec:cancellation} is devoted to showing that all AF-groupoids have cancellation, something which is needed several times in the proof of the main result. We point out that this cancellation result may be of independent interest. Then in Section~\ref{sec:LES} we present two long exact sequences in ample groupoid homology. One of them relates the homology of a groupoid equipped with a cocycle with that of the associated skew product.  The other relates the homology of restrictions to nested invariant subsets.

Both of these long exact sequences are applied to graph groupoids in Section~\ref{sec:graphhom}.  This allows us to relate the homology of a graph groupoid $\mathcal{G}_E$ with both the skew product~$\G_E \times_{c_E} \ZZ$ and the kernel $\mathcal{H}_E$. As the latter two are AF-groupoids, this truncates the long exact sequences to finite exact sequences. After some work, we obtain the embeddings~$H_1(\mathcal{G}_E) \hookrightarrow H_0(\mathcal{H}_E) \hookrightarrow H_0(\G_E\times_{c_E} \ZZ)$. In particular, we identify~$H_1(\mathcal{G}_E)$ with~$\ker(\id - \varphi)$, where $\varphi$ is an endomorphism of $H_0(\mathcal{H}_E)$ given by ``extending paths backwards''. We have to do some extra work here because we cannot deduce that~$H_0(\mathcal{H}_E) \cong H_0(\G_E\times_{c_E} \ZZ)$, as one can for finite graphs. In Section~\ref{sec:indexmap} we associate each element $\alpha$ in the topological full group $\llbracket \mathcal{G}_E \rrbracket$ with a finite clopen partition of the unit space $\mathcal{G}_E^{(0)}$. This partition is then used to give a description of the value $I(\alpha)$ of the index map under the correspondence $H_1(\mathcal{G}_E) \cong \ker(\id - \varphi)$ from the previous section.

The proof of our main result, Theorem~\ref{thm:AHintro}, is given in Section~\ref{sec:TR}. We begin the section by proving a technical lemma which plays a similar role as mixing of the shift space does in Matui's proof for SFT-groupoids. The way it is used in our proof, however, is quite different from the way mixing is used. Next we show that the assumptions in said lemma can always be arranged, by appealing to the geometric moves on graphs from the classification program of unital graph \mbox{$C^*$-algebra}~\cite{ERRS}. After that we prove that strongly connected graphs with infinite emitters have Property~TR. The proof is quite long and draws upon all of the preceding sections. By combining Matui's result for strongly connected finite graphs with our result for infinite graphs, together with another geometric move on graphs, we deduce that the AH~conjecture holds for all graph groupoids satisfying the assumptions in the AH~conjecture.

We end the paper with Section~\ref{sec:examples} where we give a couple of examples and obtain some consequences of the AH~conjecture. In particular, we consider the canonical graph groupoid model of $\mathcal{O}_\infty$ and observe that either the topological full group $\llbracket \mathcal{G}_{E_\infty} \rrbracket$ is simple or $\mathcal{G}_{E_\infty}$ has the strong AH~property, but not both. In fact, these two properties are shown to be mutually exclusive whenever the graph has an infinite emitter. This is in contrast to the case of finite graphs, where one can have both. We also observe that when $E$ has an infinite emitter, then $\llbracket \mathcal{G}_E \rrbracket$ is not finitely generated. A partial description of the abelianization $\llbracket \mathcal{G}_E \rrbracket_{\text{ab}}$ is also given in terms of the first two homology groups.

\section{Étale groupoids}\label{sec:etale}

In this section we will collect the basic notions regarding étale groupoids that we will need, as well as establish notation and conventions. Two standard references for étale groupoids (and their \mbox{$C^*$-algebras}) are Renault's thesis~\cite{Ren80} and Paterson's book~\cite{Pat}. More recent accounts are found in e.g.~\cite{Exel} and~\cite{Sims}.

If two sets $A$ and $B$ are disjoint we will denote their union by $A \sqcup B$ when we wish to emphasize that they are disjoint. When we write $C = A \sqcup B$ we mean that $C = A \cup B$ and that $A$ and $B$ are disjoint sets.

\subsection{Topological groupoids}

A \emph{groupoid} is a set $\GG$ equipped with a partially defined product~$\mathcal{G}^{(2)} \to \GG$ denoted $(g,h) \mapsto gh$, where $\mathcal{G}^{(2)} \subseteq \GG \times \GG$ is the set of \emph{composable pairs}, and an everywhere defined involutive inverse  $g \mapsto g^{-1}$ satisfying the following axioms:
\begin{enumerate}
\item If $(g_1,g_2), (g_2,g_3) \in \mathcal{G}^{(2)}$, then $(g_1 g_2,g_3), (g_1, g_2 g_3) \in \mathcal{G}^{(2)}$ and $ (g_1 g_2) g_3 = g_1 (g_2 g_3)$.
\item For all $g \in \G$, we have $(g, g^{-1}), (g^{-1}, g) \in \mathcal{G}^{(2)}$.
\item If $(g,h) \in \mathcal{G}^{(2)}$, then $g h h^{-1} = g$ and $g^{-1} g h = h$.
\end{enumerate}
The set $\GG^{(0)} \coloneqq \{g g^{-1} \mid g \in \mathcal{G} \}$ is called the \emph{unit space}, and the maps $r,s \colon \mathcal{G} \to \mathcal{G}^{(0)}$ given by $r(g) = gg^{-1}$ and $s(g) = g^{-1}g$ are called the  \emph{range} and \emph{source} maps, respectively.

If~$\mathcal{G}$ is given a topology in which the product and inverse map are continuous we call~$\mathcal{G}$ a topological groupoid. A topological groupoid is \emph{étale} if it has a locally compact topology in which the unit space is open and Hausdorff, and the range and source maps are local homeomorphisms. For the most part we will be dealing with étale groupoids which are (globally) Hausdorff, and then $\mathcal{G}^{(0)}$ is clopen in $\mathcal{G}$. We say that an étale groupoid~$\mathcal{G}$ is \emph{ample} if $\mathcal{G}^{(0)}$ is zero-dimensional, i.e.\ admits a basis of compact open sets. Étale groupoids are characterized by admitting a basis of \emph{bisections} (defined below), and ample groupoids by admitting a basis of \emph{compact bisections}.


For a subset $A \subseteq \GG^{(0)}$ we set $\mathcal{G}^A \coloneqq \{g \in \mathcal{G} \mid r(g) \in A \}$ and $\mathcal{G}_A \coloneqq \{g \in \mathcal{G} \mid s(g) \in A \}$. For singleton sets $A = \{x\}$ we drop the braces and write $\mathcal{G}^x$ and $\mathcal{G}_x$, respectively. The \emph{isotropy group} of $x\in \GG^{(0)}$ is $\GG_x^x \coloneqq \mathcal{G}^x \cap \mathcal{G}_x$, and the \emph{isotropy} of $\mathcal{G}$ is~$\GG' \coloneqq \bigsqcup_{x \in \mathcal{G}^{(0)}} \GG_x^x$. We say that $\mathcal{G}$ is \emph{principal} if $\GG' = \mathcal{G}^{(0)}$, and \emph{effective}\footnote{We remark that the literature is not entirely consistent regarding this notion. For example in~\cite{MatTFG} the term \emph{essentially principal} is used. The term \emph{topologically principal} also appear in the literature, but this usually refers to a slightly stronger notion.} if the interior of~$\GG'$ equals $\mathcal{G}^{(0)}$. The~\emph{\mbox{$\GG$-orbit}} of a unit $x$ is the set $\orb_{\GG}(x) \coloneqq s(\mathcal{G}^x) = r(\mathcal{G}_x)$. We call $\GG$ \emph{minimal} when every~\mbox{$\GG$-orbit} is dense in $\GG^{(0)}$. This is equivalent to there being no nontrivial open (or closed)~\emph{\mbox{$\mathcal{G}$-invariant}} subsets $A \subseteq \mathcal{G}^{(0)}$, meaning that $\mathcal{G}^A = \mathcal{G}_A$. The \emph{restriction} of $\GG$ to~$A$ is $\mathcal{G} \vert_A \coloneqq \mathcal{G}^A \cap \mathcal{G}_A$, and this is a subgroupoid of $\mathcal{G}$ with unit space $A$. If $A$ is open and $\mathcal{G}$ is étale, then $\GG_{|A}$ is an open étale subgroupoid of $\mathcal{G}$. We say that $A$ is \emph{$\mathcal{G}$-full} if~$r(\mathcal{G}_A) = \mathcal{G}^{(0)}$, in other words if $A$ intersects every $\mathcal{G}$-orbit. Two étale groupoids~$\mathcal{G}$ and~$\mathcal{H}$ are \emph{Kakutani equivalent} if there exists a $\mathcal{G}$-full clopen subset~$A \subseteq \GG^{(0)}$ and an $\mathcal{H}$-full clopen subset~$B \subseteq \mathcal{H}^{(0)}$ such that~$\mathcal{G} \vert_A \cong \mathcal{H} \vert_B$ (as topological groupoids). This notion of groupoid equivalence admits many different descriptions, see~\cite[Theorem~3.12]{FKPS}.

\subsection{The topological full group}\label{subsec:tfg}

An open subset $U \subseteq \mathcal{G}$ of an étale groupoid $\mathcal{G}$ is called a \emph{bisection} if both $r$ and $s$ are injective on $U$. It follows then that~$r \vert_U \colon U \to r(U)$ is a homeomorphism, and similarly for~$s$. Thus we get a homeomorphism $\pi_U \coloneqq r_{\vert U} \circ (s_{\vert U})^{-1}$ from $s(U)$ to $r(U)$ which maps $s(g)$ to $r(g)$ for each $g \in U$. We say that the bisection~$U$ is \emph{full} if $r(U) = s(U) = \mathcal{G}^{(0)}$, and in this case~ $\pi_U$ is a homeomorphism of $\mathcal{G}^{(0)}$. For a homeomorphism $\alpha \colon X \to X$ of a topological space~$X$ we define the \emph{support} of $\alpha$ to be the set $\supp(\alpha) \coloneqq \overline{ \{x \in X \mid \alpha(x) \neq x\}}$.

The \emph{topological full group} of an effective étale groupoid $\mathcal{G}$ is
\[\llbracket \mathcal{G} \rrbracket \coloneqq \{ \pi_U \mid U \subseteq \mathcal{G} \text{ full bisection} \ \& \ \supp(\pi_U) \text{ is compact}  \}, \]
which is a subgroup of the homeomorphism group of $\mathcal{G}^{(0)}$. The commutator subgroup of $\llbracket \mathcal{G} \rrbracket$ is denoted by~$\DD(\llbracket \mathcal{G} \rrbracket)$. We remark that when $\mathcal{G}$ is effective and Hausdorff, then~$\supp(\pi_U)$ is also open for any full bisection $U$. And if~$V \neq U$ are different bisections, then $\pi_U \neq \pi_V$. As a notational remark, if we are given an element~$\alpha \in \llbracket \mathcal{G} \rrbracket$ we let $U_\alpha$ denote the unique full bisection which gives rise to $\alpha$, i.e. the one with~$\alpha = \pi_{U_\alpha}$.

The following construction will be used several times. Suppose $U \subseteq \mathcal{G}$ is a compact bisection with $r(U) \cap s(U) = \emptyset$. Define
\[\widehat{U} \coloneqq U \sqcup U^{-1} \sqcup \left( \mathcal{G}^{(0)} \setminus (r(U) \cup s(U)) 	   \right).\]
Then~$\widehat{U}$ is a full bisection and its associated homeomorphism $\pi_{\widehat{U}}$ satisfies 
\[\pi_{\widehat{U}}(s(U)) = r(U), \quad \pi_{\widehat{U}}(r(U)) = s(U), \quad \supp(\pi_{\widehat{U}}) = r(U) \cup s(U), \quad \left( \pi_{\widehat{U}}\right)^2 = \id_{\mathcal{G}^{(0)}}. \]
It is clear that $\pi_{\widehat{U}} \in \llbracket \mathcal{G} \rrbracket$. If~$\tau \in \llbracket \mathcal{G} \rrbracket$ satisfies $\tau^2 = 1$ and the set $\{x \in \mathcal{G}^{(0)} \mid \tau(x) = x\}$ is clopen, then one can show that ~$\tau = \pi_{\widehat{U}}$ for some compact bisection $U$ as above. Following~\cite{MatTFG},~\cite{MatProd} we call these elements \emph{transpositions}. We let $\mathcal{S}(\mathcal{G})$ denote the (normal) subgroup of~$\llbracket \mathcal{G} \rrbracket$ generated by all transpositions, as in~\cite{Nek19}.   

\begin{remark}
Some authors define the topological full group to consist of the full bisections themselves, rather than their associated homeomorphisms, but for effective groupoids this is merely a matter of taste. Topological full groups are quite interesting objects in their own right and we refer to~\cite{MatSurvey} and~\cite{NO} and the references therein for more details on the subject.
\end{remark}

\subsection{Homology for ample groupoids} Let us for an ample Hausdorff groupoid~$\mathcal{G}$ describe its homology with values in $\ZZ$, as popularized by Matui in~\cite{MatHom} building on the general theory of~\cite{CM}. See also~\cite[Section~4]{FKPS} for an excellent account.

For a locally compact Hausdorff space $X$, let $C_c(X, \ZZ)$ denote the compactly supported continuous $\ZZ$-valued functions on $X$. A local homeomorphism $\psi \colon X \to Y$ between such spaces induces a homomorphism $\psi_* \colon C_c(X, \ZZ) \to C_c(Y, \ZZ)$ which is given by~${\psi_*(f)(y) = \sum_{x \in \psi^{-1}(y)} f(x)}$ for $f \in C_c(X, \ZZ)$. Only finitely many terms are nonzero in this sum. 

For $n \geq 1$, let $\mathcal{G}^{(n)}$ denote the space of composable strings of~$n$ elements from $\mathcal{G}$, equipped with the relative topology induced by the product topology on $n$ copies of $\mathcal{G}$. In particular, $\mathcal{G}^{(2)}$ is the composable pairs,  $\mathcal{G}^{(1)} = \mathcal{G}$ and for $n = 0$, we have the unit space~$\mathcal{G}^{(0)}$. Define local homeomorphisms $d_i \colon \mathcal{G}^{(n)} \to \mathcal{G}^{(n-1)}$ for $n \geq 2$ and $i = 0, \ldots, n$ by 
\[d_i(g_1, g_2, \ldots, g_n) = \begin{cases} (g_2, g_3, \ldots, g_n) &\text{if } i = 0, \\
(g_1, \ldots, g_{i-1}, g_i g_{i+1}, g_{i+2}, \ldots, g_n) &\text{if } 1 \leq i \leq n-1, \\
(g_1, g_2, \ldots, g_{n-1}) &\text{if } i = n.
\end{cases}  \]
From these we in turn define homomorphisms $\delta_n \colon C_c(\mathcal{G}^{(n)}, \ZZ) \to C_c(\mathcal{G}^{(n-1)}, \ZZ)$ by setting~$\delta_n = \sum_{i=0}^n (-1)^i (d_i)_*$, and for $n=1$ set $\delta_1 = s_* - r_*$. Then 
\begin{equation}\label{eq:chaincomplex}
\begin{tikzcd}
0  &  C_c(\mathcal{G}^{(0)}, \ZZ) \arrow{l} & C_c(\mathcal{G}^{(1)}, \ZZ) \arrow{l}[swap]{\delta_1} & C_c(\mathcal{G}^{(2)}, \ZZ) \arrow{l}[swap]{\delta_2} & \cdots \arrow{l}[swap]{\delta_3}
\end{tikzcd}
\end{equation}
becomes a chain complex and the homology groups $H_n(\mathcal{G})$ is defined as the homology of this complex, i.e.\ $H_n(\mathcal{G}) = \ker \delta_n / \im \delta_{n+1}$. We will use $C_\bullet(\G, \ZZ)$ to denote the chain complex~\eqref{eq:chaincomplex}.


Since the zeroth and first homology groups will appear frequently in this text, by virtue of being ingredients in the AH~conjecture, we describe the two homomorphisms $\delta_1$ and $\delta_2$ that define them in more detail. The former is the difference of the maps from $C_c(\mathcal{G}, \ZZ)$ to $C_c(\mathcal{G}^{(0)}, \ZZ)$ induced by the source and range maps, and these are in turn given by 
\[s_*(f)(x) = \sum_{g \in \mathcal{G}_x} f(g) \quad  \text{and} \quad r_*(f)(x) = \sum_{g \in \mathcal{G}^x} f(g) \]
for $f \in C_c(\mathcal{G}, \ZZ)$ and $x \in \mathcal{G}^{(0)}$. As for the latter we have that $\delta_2 = (d_0)_* - (d_1)_* + (d_2)_*$, where each of these summands are maps from $C_c(\mathcal{G}^{(2)}, \ZZ)$ to $C_c(\mathcal{G}, \ZZ)$ given by
\begin{align*}
(d_0)_*(\psi)(g) &= \sum_{h \in \mathcal{G}, \ s(h) = r(g) } \psi(h, g) \\
(d_1)_*(\psi)(g) &= \sum_{(h_1, h_2) \in \mathcal{G}^{(2)}, \ h_1 h_2 = g   } \psi(h_1, h_2) \\
(d_2)_*(\psi)(g) &= \sum_{h \in \mathcal{G}, \ r(h) = s(g) } \psi(g, h)
\end{align*}
for $\psi \in C_c(\mathcal{G}^{(2)}, \ZZ)$ and $g \in \mathcal{G}$. 

Observe that $H_0$ is spanned (over $\ZZ$) by equivalence classes of indicator functions of compact open subsets of the unit space. For any compact bisection~$U \subseteq \mathcal{G}$ we have~$\left[1_{s(U)}\right] = \left[1_{r(U)}\right]$ in $H_0(\mathcal{G})$, since $\delta_1(1_U) = 1_{s(U)} - 1_{r(U)}$. If we view a compact open set~$A \subseteq \mathcal{G}^{(0)}$ as a subset of $\mathcal{G}$, then $1_A \in \ker \delta_1$ and $[1_A] = 0$ in $H_1(\mathcal{G})$ since $\delta_2(1_{\Delta A}) = 1_A$, where $\Delta A \subseteq \mathcal{G}^{(2)}$ denotes the diagonal in~$A \times A$. 

Any étale homomorphism\footnote{That is, a local homeomorphism which respects the groupoid structures.} $\rho \colon \mathcal{G} \to \mathcal{H}$ induce local homeomorphisms \mbox{$\rho^{(n)} \colon \mathcal{G}^{(n)} \to \mathcal{H}^{(n)}$} for $n \geq 0$ by applying $\rho$ in each coordinate. The induced maps $(\rho^{(n)})_*$ from $C_c(\mathcal{G}^{(n)}, \ZZ)$ to~$C_c(\mathcal{H}^{(n)}, \ZZ)$ form a chain map $\rho_\bullet \colon C_\bullet(\G, \ZZ) \to C_\bullet(\mathcal{H}, \ZZ)$ which in turn induce homomorphisms $H_n(\rho_\bullet) \colon H_n(\mathcal{G}) \to H_n(\mathcal{H})$. This assignment is functorial. In particular, if $\mathcal{G} \subseteq \mathcal{H}$ is an open subgroupoid, then the inclusion map $\iota \colon \mathcal{G} \to \mathcal{H}$ induce homomorphisms $H_n(\iota_\bullet) \colon H_n(\mathcal{G}) \to H_n(\mathcal{H})$ given by $[1_W] \mapsto [1_W]$ for any compact open set~${W \subseteq \mathcal{G}^{(n)}}$. And if $Y \subseteq \mathcal{G}^{(0)}$ is a $\mathcal{G}$-full clopen, then the inclusion map $\iota$ induce isomorphisms $H_n(\iota_\bullet) \colon H_n(\mathcal{G} \vert_Y) \xrightarrow{\ \cong \ } H_n(\mathcal{G})$ for all $n \geq 0$~\cite[Lemma~4.3]{FKPS}. From this it is clear that Kakutani equivalent groupoids have the same homology. 

When $n=0$ in the setting above the inverse map $H_0(\iota_\bullet)^{-1} \colon H_0(\mathcal{G}) \to H_0(\mathcal{G} \vert_Y)$ can be described as follows. Let $A \subseteq \mathcal{G}^{(0)}$ be a compact open set. By fullness of~$Y$, we can for each $x \in A$ find a compact bisection $U_x \subseteq \mathcal{G}$ with $x \in s(U_x) \subseteq A$ and ${r(U_x) \subseteq Y}$. By compactness and $0$-dimensionality we can find finitely many compact bisections~$U_1, \ldots, U_m$ so that the $s(U_i)$'s form a clopen partition of $A$ and so that~$r(U_i) \subseteq Y$. Now $[1_A] = \sum_{i=1}^m [1_{s(U_i)}] = \sum_{i=1}^m [1_{r(U_i)}]$ in $H_0(\mathcal{G})$, and we thus have
\begin{equation}\label{eq:0inverse}
H_0(\iota_\bullet)^{-1}([1_A]) = \sum_{i=1}^m [1_{r(U_i)}] \in H_0(\mathcal{G} \vert_Y).
\end{equation}

\subsection{AF-groupoids and their homology}
Let $\mathcal{R}_n$ denote the full equivalence relation on the finite set $\{1,2, \ldots, n\}$, viewed as a discrete groupoid. When $X$ is a locally compact Hausdorff space, Renault~\cite{Ren80} calls the product groupoid $X \times \mathcal{R}_n$ an \emph{elementary groupoid of type~$n$}, where we view $X$ as a trivial groupoid $X = X^{(0)}$. We will call an étale groupoid~$\mathcal{G}$ \emph{elementary} if it is Hausdorff, principal and~$\mathcal{G} \setminus \mathcal{G}^{(0)}$ is compact. Lemma~3.4 in~\cite{GPS04} shows that an ample elementary groupoid is isomorphic to a finite disjoint union of elementary groupoids of type $n_i$. An \emph{AF-groupoid} is an ample groupoid which can be written as an increasing union of open elementary subgroupoids.




It is a well known fact that when $\mathcal{G}$ is an AF-groupoid, its homology is given by 
\[H_n(\mathcal{G}) \cong \begin{cases}
K_0(C^*_r(\mathcal{G})) \quad &n=0, \\
0 \quad &n \geq 1,
\end{cases} \]
where $C^*_r(\mathcal{G})$ denotes the reduced groupoid \mbox{$C^*$-algebra} of $\mathcal{G}$, which in this case is an AF-algebra. The $H_0$-group (and the $K_0$-group) coincides with the dimension group of any defining Bratteli diagram (as an ordered abelian group with distuingished order unit). Stated like this it first appeared in~\cite{MatHom} (for compact unit spaces), but it can be traced back to the earlier works~~\cite{Ren80} and~\cite{Kri}. The case of a non-compact unit space is treated in~\cite{FKPS}. 

\begin{theorem}[{\cite[Corollary~5.2]{FKPS}}] \label{thm:KH}
Let $\mathcal{G}$ be an AF-groupoid. Then the map $[1_A]_{H_0} \mapsto [1_A]_{K_0}$ for~$A \subseteq \mathcal{G}^{(0)}$ compact open induces an isomorphism $H_0(\mathcal{G}) \cong K_0(C_r^*(\mathcal{G}))$.
\end{theorem}

\subsection{Cocycles and skew products}

When $\mathcal{G}$ is an étale groupoid and $\Gamma$ is a discrete group, we call $c \colon \mathcal{G} \to \Gamma$ a \emph{cocycle} if it is a continuous groupoid homomorphism. We shall be dealing exclusively with $\ZZ$-valued cocycles, as these are the ones that appear naturally for graph groupoids.

\begin{definition}
Let $\mathcal{G}$ be an étale groupoid with a cocycle $c \colon \mathcal{G} \to \ZZ$. The \emph{skew product groupoid} of $\mathcal{G}$ by $c$ is the groupoid $\G\times_c \ZZ \coloneqq \G\times \ZZ$ with operations 
\[(g,k)(g',m+c(g)) \coloneqq (gg',m) \quad \text{and} \quad (g,m)^{-1} \coloneqq (g^{-1}, m + c(g)), \]
so that $s(g,m)=(s(g),c(g)+m)$ and $r(g,m)=(r(g),m)$.
\end{definition} 

The skew product groupoid becomes an étale groupoid in the product topology. The unit space of $\G\times_c \ZZ$ can be identified with $\mathcal{G}^{(0)} \times \ZZ$. And for each bisection $U \subseteq \mathcal{G}$ and~$m \in \ZZ$, the set $U \times \{m\}$ is a bisection in $\G\times_c \ZZ$. We record the following elementary lemma about the kernel of the cocycle sitting inside the skew product. 

\begin{lemma}\label{lem:kerc}
Let $\mathcal{G}$ be an étale groupoid with a cocycle $c \colon \mathcal{G} \to \ZZ$. Then $\ker(c)$ is a clopen subgroupoid of $\mathcal{G}$, and we have  $\left( \G\times_c \ZZ \right) \vert_{\mathcal{G}^{(0)} \times \{0\}} \cong \ker(c)$ via the map $(g,0) \mapsto g$.
\end{lemma}

\begin{remark}
We emphasize that even though $\ker(c)$ is a clopen subgroupoid of $\mathcal{G}$, and embeds as a clopen subgroupoid of the skew product $\G\times_c \ZZ$, we can generally not embed~$\mathcal{G}$ itself into $\G\times_c \ZZ$ in any way (e.g.~$\G\times_c \ZZ$ can be principal while~$\mathcal{G}$ is not.)
\end{remark}

There is a canonical action $\widehat{c}$ by $\ZZ$ on $\G\times_c \ZZ$ defined by~${\widehat{c}_k \cdot (g,m) = (g, m+k)}$, i.e.\ shifting the integer coordinate. If one then forms the semi-direct product groupoid ${(\G\times_c \ZZ) \rtimes_{\widehat{c}} \ZZ}$, one gets that this semi-direct product is Kakutani equivalent to the groupoid $\G$ that we started with, and hence they have the same homology groups~\cite{MatHom}. This is what Matui uses when he computes the homology groups of $\mathcal{G}_E$ for a finite graph~$E$ by means of a spectral sequence~\cite{MatTFG}. We shall instead use a long exact sequence in homology from~\cite{Ort}, to be described in Section~\ref{sec:LES}.

\section{Graphs and their groupoids}\label{sec:graphs}

As this paper primarily concerns graph groupoids, we spend some time in this section recalling their definition and properties, as well as establishing notation. We refer to~\cite{BCW} and~\cite{NO} for additional details.

\subsection{Graphs} A \emph{(directed) graph} ${E=(E^0,E^1,r,s)}$ consists of two countable sets $E^0$ and $E^1$, whose elements are called vertices and edges, respectively, in addition to range and source maps $r,s \colon E^1 \to E^0$. We say that $E$ is \emph{finite} if both $E^0$ and $E^1$ are finite sets.

A \emph{path} is a sequence of edges $\mu=e_1 e_2 \ldots e_n$ such that $r(e_i)=s(e_{i+1})$ for~${1\leq i\leq n-1}$. The \emph{length} of $\mu$ is $\left| \mu \right| \coloneqq n$. The set of paths of length $n$ is denoted~$E^n$ and the set of all finite paths is~${E^* \coloneqq \bigcup_{n=0}^\infty E^n}$. The range and source maps extend to $E^*$ by setting~${r(\mu) = r(e_n)}$ and~$s(\mu) = s(e_1)$. For $v \in E^0$, we set~$s(v) = r(v) = v$. If $\mu, \nu \in E^*$ satisfy $r(\mu) = s(\nu)$, then $\mu \nu \in E^*$ denotes their concatenation. We say that $\mu$ is a \emph{subpath} of $\nu$ if $\nu = \mu \lambda$ for some path $\lambda$ with $s(\lambda) = r(\mu)$. Two paths are called \emph{disjoint} if neither is a subpath of the other. A graph $E$ is called \emph{strongly connected} if for each pair of vertices $v,w \in E^0$ there is a path from $v$ to $w$. By a \emph{strongly connected component} we mean a maximal subset of vertices such that there is a path between any two vertices in this subset. The strongly connected components form a partition of $E^0$.

An edge $e \in E^1$ with $r(e) = s(e)$ is called a \emph{loop}. More generally, a \emph{cycle} is a nontrivial path $\mu$ (i.e.\ $\left| \mu \right| \geq 1$) with $r(\mu) = s(\mu)$, and we say that $\mu$ is \emph{based} at $s(\mu)$ or that $s(\mu)$ \emph{supports} the cycle $\mu$. By $\mu^k$ we mean $\mu$ concatenated $k$ times. A graph is called \emph{acyclic} if it has no cycles. An \emph{exit} for a path $\mu = e_1 \ldots e_n$ is an edge~$e \in E^1$ such that~$s(e) = s(e_i)$ and $e \neq e_i$ for some $1 \leq i \leq n$. The graph $E$ is said to satisfy \emph{Condition~(L)} if every cycle in $E$ has an exit. 

For a vertex $v\in E^0$ and $n \geq 1$ we define the sets $vE^n \coloneqq \{ \mu \in E^n \mid s(\mu) = v\}$ and~${E^n v \coloneqq \{ \mu \in E^n \mid r(\mu) = v\}}$. We call $v$ a \emph{sink} if $vE^1 = \emptyset$ and a \emph{source} if ${E^1 v = \emptyset}$. Furthermore, $v$ is called an \emph{infinite emitter} if $vE^1$ is an infinite set. Sinks and infinite emitters are collectively referred to as \emph{singular} vertices and the set of these is denoted~$E^0_{\text{sing}}$. Non-singular vertices are called \emph{regular}. A graph is \emph{row-finite} if it has no infinite emitters, and \emph{essential} if it has no sinks nor sources.

\subsection{The boundary path space} An \emph{infinite path} in a graph $E$ is a sequence of edges $x = e_1 e_2 e_3 \ldots$ such that $r(e_i)=s(e_{i+1})$ for all $i \in \mathbb{N}$. We define~${s(x) \coloneqq s(e_1)}$ and~${\left| x \right| \coloneqq \infty}$. The set of all infinite paths is denoted $E^\infty$. We call $E$ \emph{cofinal} if for every vertex $v \in E^0$ and for every infinite path $e_1 e_2 \ldots \in E^\infty$, there is a path from $v$ to $s(e_n)$ for some ~$n \in \mathbb{N}$. The \emph{boundary path space} of $E$ is
\[\partial E \coloneqq E^\infty \cup \{\mu\in E^* \mid r(\mu)\in E^0_{\text{sing}}\}.\]
The \emph{cylinder set} of a finite path $\mu \in E^*$ is~$Z(\mu) \coloneqq \{ \mu x \mid x \in \partial E, \ s(x) = r(\mu) \}$. Given a finite subset $F \subseteq r(\mu)E^1$, we  define the associated \emph{punctured cylinder set} to be~$Z(\mu \setminus F) \coloneqq Z(\mu) \setminus \left( \bigsqcup_{e\in F} Z(\mu e) \right)$. Note that two finite paths are disjoint if and only if their cylinder sets are disjoint sets.

The topology on the boundary path space $\partial E$ is specified by the countable basis~${\left\{Z(\mu \setminus F)  \mid \mu \in E^*, F \subseteq_{\text{finite}} r(\mu)E^1 \right\}}$. This turns $\partial E$ into a locally compact Hausdorff space in which each basic set~$Z(\mu \setminus F)$ is compact open~\cite{Web}. Note that the boundary path space $\partial E$ itself is compact if and only if $E^0$ is finite. Existence of isolated points in $\partial E$ is characterized in~\cite[Section~3]{CW}.  

Define $\partial E^{\geq n} \coloneqq \{ x \in \partial E  \mid \left| x \right| \geq n \}$ for $n \in \mathbb{N}$, which are open subsets of $\partial E$. The \emph{shift map} on $E$ is the map $\sigma_E \colon \partial E^{\geq 1} \to \partial E$ given by $\sigma_E(e_1 e_2 e_3 \ldots) = e_2 e_3 e_4 \ldots$ for~$e_1 e_2 e_3 \ldots \in \partial E^{\geq 2}$ and $\sigma_E(e) = r(e)$ for $e \in \partial E \cap E^1$. The image $\sigma_E \left( \partial E^{\geq 1} \right)$ is also open in $\partial E$ and the shift map is surjective precisely when $E$ has no sources. We also set~$\sigma_E^0 = \id_{\partial E}$. Then the iterates $\sigma_E^n \colon \partial E^{\geq n}\to\partial E$ are local homeomorphisms for each~$n \geq 0$.

\subsection{Graph groupoids} The \emph{graph groupoid} of a graph $E$ is
\[\mathcal{G}_E \coloneqq \{(x,m-n,y) \mid m,n \geq 0, \ x \in \partial E^{\geq m}, \ y \in \partial E^{\geq n}, \ \sigma_E^m(x) = \sigma_E^n(y) \}, \]
equipped with the product $(x,k,y) \cdot (y,l,z) \coloneqq (x,k+l,z)$ (and undefined otherwise), and inverse $(x,k,y)^{-1} \coloneqq (y,-k,x)$. In other words, a triplet $(x,k,y) \in \partial E \times \ZZ \times \partial E$ belongs to the graph groupoid $\mathcal{G}_E$ if and only if $x = \mu z$ and $y = \nu z$ for some finite paths~$\mu, \nu \in E^*$ and a boundary path $z \in \partial E$ satisfying $\vert \mu \vert = \vert \nu \vert + k$.

Given two finite paths $\mu, \nu \in E^*$ with $r(\mu) = r(\nu)$ and a finite subset $F \subseteq r(\mu)  E^1$ we define the associated \emph{punctured double cylinder set} to be the following subset of $\mathcal{G}_E$:
\[Z(\mu, F, \nu) \coloneqq \{(x, \vert \mu \vert - \vert \nu \vert,y) \mid x\in Z(\mu \setminus F), \ y\in Z(\nu \setminus F), \ \sigma_E^{\vert \mu \vert  }(x) = \sigma_E^{\vert \nu \vert  }(y) \}.  \]
Equipping the graph groupoid $\mathcal{G}_E$ with the topology generated by the countable basis
\[\left\{ Z(\mu, F, \nu) \mid \mu, \nu \in E^*, r(\mu) = r(\nu), F \subseteq_{\text{finite}} r(\mu)E^1 \right\}\] turns it into an ample Hausdorff groupoid, as each $Z(\mu, F, \nu)$ becomes a compact open bisection. That this indeed is the standard topology on $\mathcal{G}_E$, as in e.g.~\cite{BCW}, was shown in~\cite[Lemma~9.2]{NO}.

The unit space of $\mathcal{G}_E$ is $\mathcal{G}^{(0)}_E = \{ (x,0,x) \mid x \in \partial E \}$, which we will freely identify with the boundary path space $\partial E$ via the homeomorphism $(x,0,x) \leftrightarrow x$. In terms of the bases we identify $Z(\mu ,F, \mu)$ with $Z(\mu \setminus F)$. The range and source maps of $\mathcal{G}_E$ then become $r(x,k,y) = x$ and $s(x,k,y) = y$. For a basic compact open bisection as above we have $r(Z(\mu, F, \nu)) = Z(\mu \setminus F)$ and $s(Z(\mu, F, \nu)) = Z(\nu \setminus F)$.

A graph groupoid $\mathcal{G}_E$ is effective precisely when $E$ satisfies Condition~(L)~\cite[Proposition~2.3]{BCW}, and $\mathcal{G}_E$ is minimal if and only if $E$ is both cofinal and there exists a path from every vertex to every singular vertex~\cite[Proposition~8.3]{NO}. On any graph groupoid there is a canonical cocycle $c_E \colon \G_E \to \ZZ$ given by $(x,k,y)\mapsto k$. We define \[\mathcal{H}_E \coloneqq \ker(c_E) = \{(x,0,y) \in \mathcal{G}_E \},\] which is a clopen subgroupoid of~$\mathcal{G}_E$. The subgroupoid $\mathcal{H}_E$ and the skew product groupoid $\G_E \times_{c_E} \ZZ$ will play important roles in the proof of the AH~conjecture for~$\mathcal{G}_E$.

The full and the reduced groupoid \mbox{$C^*$-algebra} of a graph groupoid coincide. There is a canonical isomorphism $C^*_r(\mathcal{G}_E) \cong C^*(E)$ which is given by mapping the indicator function $1_{Z(v,v)} \in C_c(\mathcal{G}_E, \mathbb{C})$ to the projection $p_v \in C^*(E)$ for each $v \in E^0$ and mapping~$1_{Z(e, r(e))} \in C_c(\mathcal{G}_E, \mathbb{C})$ to the partial isometry $s_e \in C^*(E)$ for each $e \in E^1$~\cite[Proposition~2.2]{BCW}. For an introduction to graph \mbox{$C^*$-algebras}, see~\cite{Rae}.

\subsection{The skew graph} 
Let $E$ be a graph. The \emph{skew graph} of $E$, denoted $E\times \ZZ$, is the graph with vertices $(E\times \ZZ)^0 = E^0\times \ZZ$ and edges $(E\times \ZZ)^1 = E^1\times\ZZ$, such that~$s(e,i)=(s(e),i)$ and $r(e,i)=(r(e),i-1)$. See Figure~\ref{fig:skew} for an example.
\begin{figure}[h]
  \[\begin{tikzpicture}[vertex/.style={circle, draw = black, fill = black, inner sep=0pt,minimum size=5pt}, implies/.style={double,double equal sign distance,-implies}]

\node at (-1.5,1) {$E$};
\node[vertex] (v) at (0,2) [label=above left:$v$] {};
\node[vertex] (w) at (0,0) [label=below left:$w$] {};

\node at (3.5,1) {$E \times \ZZ$};
\node at (5,1) {$\cdots$};
\node[vertex] (v-1) at (6,2) [label=above:{$(v, -1)$}] {}; 
\node[vertex] (w-1) at (6,0) [label=below:{$(w,-1)$}] {};
\node[vertex] (v0) at (8,2) [label=above:{$(v, 0)$}] {};
\node[vertex] (w0) at (8,0) [label=below:{$(w,0)$}] {};
\node[vertex] (v1) at (10,2) [label=above:{$(v, 1)$}] {};
\node[vertex] (w1) at (10,0) [label=below:{$(w,1)$}] {};
\node at (11,1) {$\cdots$};

\path (v) edge[thick, bend right = 50, decoration={markings, mark=at position 0.99 with {\arrow{triangle 45}}}, postaction={decorate} ] (w)

(w) edge[thick, implies]  (v)
      edge[thick, loop, min distance = 20mm, looseness = 10, out = 45, in = 315, implies]  (w)

(w1) edge[thick, implies]  (w0)	edge[thick, implies]  (v0)

(w0) edge[thick, implies]  (w-1)	edge[thick, implies]  (v-1)  

(v1) edge[thick, decoration={markings, mark=at position 0.99 with {\arrow{triangle 45}}}, postaction={decorate} ] (w0)

(v0) edge[thick, decoration={markings, mark=at position 0.99 with {\arrow{triangle 45}}}, postaction={decorate} ] (w-1)           ; 
\end{tikzpicture}\]
      \caption{An example of a graph and its skew graph. A double arrow indicates that there are infinitely many edges.}
      \label{fig:skew}
\end{figure}
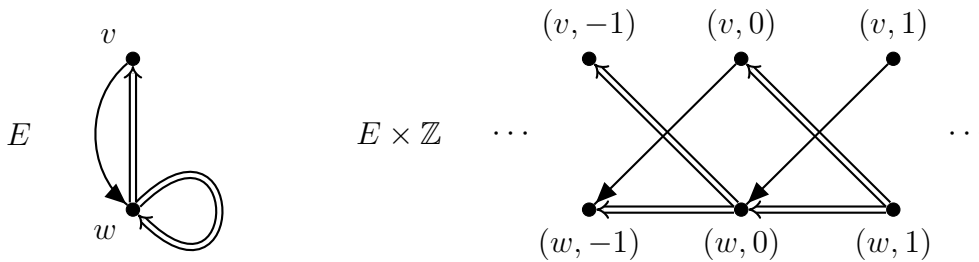

The skew graph $E\times \ZZ$ played a part in the computation of $K$-theory for graph \mbox{$C^*$-algebras}~\cite{RS}. A useful fact is that the skew graph  is always acyclic, and therefore its graph \mbox{$C^*$-algebra}, $C^*(E\times \ZZ)$, is an AF-algebra~\cite[Corollary~2.13]{DT05}. Thus its~$K_1$ group vanishes, which in turn allows the $K$-theory of $C^*(E)$ to be computed from a suitable six-term exact sequence which relates the $K$-theory of the skew graph \mbox{$C^*$-algebra} with that of the original graph \mbox{$C^*$-algebra}. 
As Matui and others have noticed, one can do something similar for graph groupoids to compute their homology, see \cite{MatHom}, \cite{Ort}, \cite{FKPS}. We will turn to this in Section~\ref{sec:graphhom}. For now, let us note that the skew graph corresponds to taking the skew product of the graph groupoid by the canonical graph cocycle.

\begin{lemma}\label{lem:skew}
For any graph $E$ we have that $\G_E\times_{c_E}\ZZ \cong \G_{E\times \ZZ}$ as étale groupoids via the map $((x,k,y),m) \mapsto (x^{(m)},k,y^{(m+k)})$, where $x^{(m)} \in \partial (E\times \ZZ)$ denotes the boundary path whose edges correspond to those in $x$, but which is  anchored at level $m$ in $E\times \ZZ$.
\end{lemma}

Throughout this paper it will be crucial that the skew product of any graph groupoid is an AF-groupoid. This was observed for finite graphs in~\cite{MatHom} and for row-finite graphs it follows from~\cite[Lemma~6.1]{FKPS}. Since we are allowing infinite emitters in our graphs, we include an argument covering the general case. 



\begin{proposition}\label{prop:AF}
Let $E$ be an acyclic graph. Then $\mathcal{G}_E$ is an AF-groupoid.
\end{proposition}
\begin{proof}
Recall that all graphs are assumed to be countable. Therefore we can  find an increasing sequence of finite subgraphs $F_1 \subseteq F_2 \subseteq F_3 \subseteq \ldots$ of $E$ such that $\cup_{n=1}^\infty F_n = E$. From these we define the following finite sets of pairs of paths
\[\mathcal{E}_n \coloneqq \{ (\mu, \nu) \in (F_n)^* \times (F_n)^* \mid r(\mu) = r(\nu)  \}. \]
We claim that the following subsets of $\mathcal{G}_E$ form an exhaustive sequence of open elementary subgroupoids: 
\[\mathcal{K}_{E,n} \coloneqq \mathcal{G}_E^{(0)} \bigcup  \bigcup_{(\mu, \nu) \in \mathcal{E}_n  } Z(\mu, \nu).  \]
A priori, it is not entirely clear that the $\mathcal{K}_{E,n}$'s are closed under multiplication (in $\mathcal{G}_E$). This relies on the acyclicity of $E$, and we provide an argument below.

 Suppose $g, h \in \mathcal{K}_{E,n}$ and that the product $g \cdot h$ is defined (i.e.\ the source of $h$ is the range of $g$). This means that $g = (\mu x, k, \nu x) \in Z(\mu, \nu)$ and $h = (\rho y, l, \tau y) \in Z(\rho, \tau)$, where $\mu, \nu, \rho, \tau$ are finite paths in $F_n$ and $\nu x = \rho y$. The latter equality implies that either $\nu \leq \rho$ or $\nu \geq \rho$. Assuming that $\nu \leq \rho$ (the other case proceeds similarly), there is a finite path $\gamma$, necessarily also in $F_n$, such that $\rho = \nu \gamma$. And then $x = \gamma y$, which means that $g \cdot h = (\mu \gamma y, k+l, \tau y)$. Since $E$ is acyclic, $\mathcal{G}_E$ is principal and therefore we must have $k+l = \vert \mu \gamma \vert - \vert \tau \vert$. This shows that $g \cdot h \in Z(\mu \gamma, \tau) \subseteq \mathcal{K}_{E,n}$, as desired. 
 
On the other hand, it is clear that $\mathcal{K}_{E,n}$ is closed under taking inverses, and hence~$\mathcal{K}_{E,n}$ is a clopen subgroupoid of $\mathcal{G}_E$. It follows from the finiteness of $\mathcal{E}_n$ that $\mathcal{K}_{E,n} \setminus \mathcal{G}_E^{(0)}$ is compact. Finally,~$\mathcal{K}_{E,n}$ is principal because $\mathcal{G}_E$ is. This shows that $\mathcal{G}_E$ is an \mbox{AF-groupoid}.
\end{proof}

Combining Lemma~\ref{lem:skew} and Proposition~\ref{prop:AF} together with the fact that $\mathcal{H}_E$ embeds as a clopen subgroupoid of $G_E\times_{c_E}\ZZ$ (Lemma~\ref{lem:kerc})  we obtain the following corollary. 

\begin{corollary}\label{cor:AF}
For any graph $E$, both $\G_E\times_{c_E}\ZZ$ and $\mathcal{H}_E$ are AF-groupoids.
\end{corollary}

We end this section by describing a consequence of Theorem~\ref{thm:KH} that we shall need in the proof of Lemma~\ref{lem:kerker}. For an arbitrary graph $E$ the $K_0$-group of its graph \mbox{$C^*$-algebra} is isomorphic to the abelian group generated by elements $g_v$ for $v \in E^0$, subject to the relations 
\[g_v = \sum_{e \in v E^1}  g_{r(e)}  \]
whenever $v$ is a regular vertex~\cite{DT02}. And this isomorphism is implemented by mapping~$[p_v]_0$ to~$g_v$, where $p_v$ denotes the projection in $C^*(E)$ associated to $v$. Using  the identification between $K_0$ and $H_0$ for AF-groupoids from Theorem~\ref{thm:KH}, together with the fact that the skew product $\G_E\times_{c_E}\ZZ$ is an AF-groupoid, we deduce the following.

\begin{lemma}\label{lem:freesummand}
Let $E$ be a graph. For each $w \in E^0_\text{sing}$ and $i \in \ZZ$, the element $\left[ 1_{Z(w) \times \{i\}} \right]$ generates a free summand of $H_0(\G_E\times_{c_E}\ZZ)$.
\end{lemma}

\section{The AH~conjecture}\label{sec:AH}


It is time to define the AH~conjecture properly, as well as discuss its current status and some aspects of how one can prove it. We will also define and discuss the HK~property.

\begin{matuiAH*}[\cite{MatProd}]
Let~${\mathcal{G}}$ be an effective minimal second countable Hausdorff étale groupoid whose unit space $\mathcal{G}^{(0)}$ is a Cantor space. Then the following sequence is exact
\begin{equation}\label{eq:AH}
\begin{tikzcd}
H_0(\mathcal{G}) \otimes \mathbb{Z}_2 \arrow{r}{j} & \llbracket \mathcal{G} \rrbracket_{\text{ab}} \arrow{r}{I_{\text{ab}}} & H_1(\mathcal{G}) \arrow{r} & 0.
\end{tikzcd} 
\end{equation}
\end{matuiAH*}

\subsection{The maps in the AH~conjecture}
Let us recall the two maps that appear in~\eqref{eq:AH}. The \emph{index map} $I \colon \llbracket \mathcal{G} \rrbracket \to H_1(\G)$ is the homomorphism given by $\pi_U\mapsto \left[1_{U} \right]$, where $U$ is a full bisection in $\mathcal{G}$. We denote the induced map on the abelianization~$\llbracket \mathcal{G} \rrbracket_{\text{ab}}$ by $I_{\text{ab}}$. The index map was introduced in the setting of Cantor minimal systems in~\cite{GPS99} and later generalized to étale groupoids over Cantor spaces in~\cite{MatHom}. 

Many of the results leading up to the main result do not require the unit space of the groupoid to be compact. And in some of these the index map appear. But the definition of the index map above does not make sense in the non-compact case. For if $\mathcal{G}$ is an ample Hausdorff groupoid with $\mathcal{G}^{(0)}$ non-compact, then any full bisection~$U \subseteq \mathcal{G}$ is non-compact as well, and so $1_U$ is not compactly supported. However, there is a straightforward way to remedy this. As shown in~\cite{NO}, where we extended the definition of the topological full group to the non-compact setting, each full bisection $U \subseteq \mathcal{G}$ can be written as
\[U = U^\perp \bigsqcup \left( \mathcal{G}^{(0)} \setminus \supp(\pi_U) \right),\]
where $U^\perp$ is a compact bisection with $s(U^\perp) = r(U^\perp) = \supp(\pi_U)$. We extend the definition of the index map by setting
\[ I(\pi_U) \coloneqq \left[1_{U^\perp} \right].\]
This agrees with the definition in the compact case because $\left[1_{U} \right] = \left[1_{U'} \right]$ if $U$ is a compact bisection which decomposes as~$U' \sqcup A$, where $A \subseteq \mathcal{G}^{(0)}$ \cite[Lemma~7.3]{MatHom}. The first homology group only ``sees'' the part of the groupoid that lies outside the unit space.

While the index map now is defined for all ample effective Hausdorff groupoids, the map~$j \colon H_0(\mathcal{G}) \otimes \mathbb{Z}_2 \to \llbracket \mathcal{G} \rrbracket_{\text{ab}}$ is a priori only defined when every $\mathcal{G}$-orbit has at least $3$ elements and $\mathcal{G}^{(0)}$ is a Cantor space. In this case, the group $ H_0(\mathcal{G}) \otimes \mathbb{Z}_2$ is generated by elements of the form $[1_{s(U)}] \otimes 1$, where $U \subseteq \mathcal{G}$ is a compact bisection with ${s(U) \cap r(U) = \emptyset}$. And the map~$j$ is given by $j([1_{s(U)}] \otimes 1) = [\pi_{\widehat{U}}] \in \llbracket \mathcal{G} \rrbracket_{\text{ab}}$, where $\pi_{\widehat{U}} \in \llbracket \mathcal{G} \rrbracket$ is the transposition defined in Subsection~\ref{subsec:tfg}. Well-definedness of this map is proved in~\cite[Section~7]{Nek19} (see also the proof of~\cite[Theorem~3.6]{MatProd}).

\subsection{The AH~conjecture for graph groupoids}
Let us determine what the assumptions in the AH~conjecture mean for graph groupoids. It follows from the results in e.g.~\cite[Section~8]{NO} that the following conditions exactly capture these assumptions.

\begin{definition}\label{def:AH}
We say that a graph $E$ satisfies the \emph{AH~criteria} if $E^0$ is finite, $E$ has no sinks, is cofinal, satisfies Condition~(L) and each vertex can reach all infinite emitters.  
\end{definition}

\begin{proposition}\label{prop:AH}
Let $E$ be a graph. Then $\mathcal{G}_E$ satisfies the assumptions in the AH~conjecture if and only if $E$ satisfies the AH~criteria.
\end{proposition}

Concretely, the AH~criteria mean that $E$ has exactly one nontrivial strongly connected component, in the sense that this is the only component which contains a cycle. In fact, there are at least two disjoint cycles based at each vertex in this component. This component also contains all infinite emitters (if there are any). Any vertex outside this component does not support a cycle, and any path from such a vertex eventually ends up in the nontrivial connected component. So if $E$ is not strongly connected, then some of the vertices outside the nontrivial connected component must be sources. Also note that $E$ is either finite or has an infinite emitter. In particular, a strongly connected graph with finitely many vertices satisfies the AH~criteria as long as it is not one of the cycle graphs~$C_n$ (i.e.\ a single cycle with $n$ vertices). 




As mentioned in the introduction, the AH~conjecture was proved for (restrictions of) graph groupoids arising from strongly connected finite graphs (which are not cycle graphs) in~\cite{MatTFG}. And the main difficulty of extending this to all graphs satisfying the AH~criteria lies in dealing with the presence of infinite emitters. Dealing with any sources in the graph, on the other hand, turns out to be quite easy. Many of the results leading up to the main result applies to more general graphs than those satisfying the AH~criteria. Therefore we will not restrict to this until the very end.

\begin{remark}
We mention in passing that, coincidentally, a graph $E$ satisfies the AH~criteria if and only if its graph \mbox{$C^*$-algebra}, $C^*(E)$, is a unital Kirchberg algebra (in the UCT~class).
\end{remark}

\subsection{Status of the AH~conjecture}
The AH~conjecture has so far been verified in a number of cases. In~\cite{MatProd} it was shown (generalizing prior results) that the AH~conjecture holds for groupoids which are almost finite and principal, and for products of SFT-groupoids. The former class includes AF-groupoids, transformation groupoids of (free) $d$-dimensional Cantor minimal systems and groupoids associated to aperiodic quasicrystals (as described in~\cite[Subsection~6.3]{Nek19}). The AH~conjecture also holds for transformation groupoids associated odometers~\cite{Scarp}. 

In some cases the map $j$ can even be shown to be injective, making~\eqref{eq:AH} a short exact sequence. When this is the case the groupoid is said to have the \emph{strong AH~property}~\cite{MatProd}. If, moreover, $j$ is split-injective, so that the sequence splits, then we say that $\mathcal{G}$ has the \emph{split AH~property}. AF-groupoids, groupoids of Cantor minimal systems ($d=1$) and SFT-groupoids all have the split AH~property~\cite[Example~4.8]{MatSurvey}. The odometers in~\cite{Scarp} have the strong AH~property, but it is unknown whether they all split. To the best of the authors' knowledge there are yet no examples which have the strong AH~property, but not the split AH~property. There are, however, examples of groupoids for which the AH~conjecture holds, yet they do not have the strong AH~property. For example groupoids arising from self-similar groups~\cite[Example~7.6]{Nek19} and products of SFT-groupoids~\cite[Subsection~5.5]{MatProd}.

\begin{remark}\label{rem:strongvssplit} 
Note that if the AH~conjecture holds for a groupoid $\mathcal{G}$ and the homology groups $H_0(\mathcal{G})$ and $H_1(\mathcal{G})$ are finitely generated, then so is the abelianization $\llbracket \mathcal{G} \rrbracket_{\text{ab}}$. And in this case, the split AH~property is equivalent to the strong AH~property together with having any isomorphism \(\llbracket \mathcal{G} \rrbracket_{\text{ab}} \cong H_1(\mathcal{G}) \oplus \left(H_0(\mathcal{G}) \otimes \mathbb{Z}_2\right).\)

We also remark that if~$H_1(\mathcal{G})$ is free abelian (i.e.\ projective in the category of abelian groups), then the split AH~property is equivalent to the strong AH~property.
\end{remark}

\subsection{The HK~property}
As mentioned in the introduction, the other conjecture from \cite{MatProd}, namely the HK~conjecture, has recently been refuted. In order to reflect this, we make the following definition for groupoids satisfying its conclusion. 

\begin{definition}\label{def:HK}
We say that an ample Hausdorff groupoid $\mathcal{G}$ has the \emph{HK~property} if there are isomorphisms 
\[ K_0\left(C^*_r(\mathcal{G})\right) \cong \bigoplus_{n=0}^\infty H_{2n}(\mathcal{G}) \quad \text{and} \quad K_1\left(C^*_r(\mathcal{G})\right) \cong \bigoplus_{n=0}^\infty H_{2n+1}(\mathcal{G}).  \]
\end{definition}

We remark that the assumptions in the HK~conjecture was exactly the same as in the AH~conjecture. As mentioned in the introdutction, the HK~property has been established for several key classes of groupoids. Furthermore, the HK~property is preserved under Kakutani equivalence. It is also preserved under products, as long as the factors are amenable, due to the Künneth formula from~\cite{MatProd}. 
Most pertinent to the present paper, however, is the fact that all graph groupoids have the HK~property (even if they are not minimal or effective). More precisely, we have the following.

\begin{theorem}\label{thm:graphhom}
Let $E$ be any graph. Then $H_0(\mathcal{G}_E)  \cong K_0(C^*(E)), \ H_1(\mathcal{G}_E)  \cong K_1(C^*(E))$ and $H_n(\mathcal{G}_E) = 0$ for $n \geq 2$. In particular, $\mathcal{G}_E$ has the HK~property.
\end{theorem} 


Theorem~\ref{thm:graphhom} was established for finite essential graphs in~\cite{MatHom}. For row-finite graphs with no sinks it follows both from the results in~\cite{Ort} and~\cite{FKPS}. In~\cite{HL} the description of $H_0(\mathcal{G}_E)$ was extended to arbitrary graphs. We add the finishing touch by noting that any graph groupoid is Kakutani equivalent to the groupoid of a row-finite graph with no sinks (namely its \emph{desingularization}~\cite{DT05}).  Since Kakutani equivalent groupoids have the same homology and their reduced groupoid \mbox{$C^*$-algebras} are Morita equivalent, the theorem follows from the aforementioned results.

The $K$-groups of graph \mbox{$C^*$-algebras} are relatively easy to compute. They are, roughly speaking, determined by the Smith normal form of the part of the adjacency matrix of $E$ which only includes edges emitted by regular vertices. The group $K_0(C^*(E))$ is a quotient of $\ZZ^{\vert E^0 \vert}$ and we have $\rank(K_0(C^*(E))) \geq \vert E^0_{\text{sing}} \vert$. On the other hand, $K_1(C^*(E))$ is free abelian and $\rank(K_1(C^*(E))) = \rank(K_0(C^*(E))) - \vert E^0_{\text{sing}} \vert $. Consult e.g.~\cite[Chapter~2.3.1]{Tom} for more details and examples.

Once we have established the AH~conjecture for graph groupoids, the fact that we can compute the homology groups allows us to say something useful about the abelianization~$\llbracket \mathcal{G}_E \rrbracket_{\text{ab}}$, also when $E$ has infinite emitters. See Section~\ref{sec:examples} for a discussion of examples and consequences of the AH~conjecture. For now we note the following.

\begin{corollary}
Let $E$ be a graph. Then $\mathcal{G}_E$ has the strong AH~property if and only if~$\mathcal{G}_E$ has the split AH~property.
\end{corollary}
\begin{proof}
As $K_1(C^*(E))$ is always free~\cite{DT02}, the assertion follows from Theorem~\ref{thm:graphhom} and Remark~\ref{rem:strongvssplit}.
\end{proof}

\subsection{Aspects of proving the AH~conjecture} 

When it comes to verifying the AH~conjecture for a groupoid $\mathcal{G}$, the hardest part is arguably to establish that $\ker(I_{\text{ab}}) \subseteq \im(j)$. Indeed, the reverse inclusion $I_{\text{ab}} \circ j = 0$ is always true, since all transpositions belong to~$\ker(I)$. That is, $\mathcal{S}(\mathcal{G}) \leq  \ker(I)$. For if $U \subseteq \mathcal{G}$ is a compact bisection with disjoint source and range, then 
\[  I \left( \pi_{\widehat{U}} \right) = [1_{\widehat{U}}] =  \left[ 1_{U \sqcup U^{-1} \sqcup (\mathcal{G}^{(0)} \setminus \supp(\pi_{\widehat{U}}))}  \right]  = \left[ 1_U +  1_{U^{-1}} \right] = 0 \in H_1(\mathcal{G}),  \]
using~\cite[Lemma~7.3]{MatHom}. Surjectivity of the index map has already been established for two general classes of groupoids, namely for \emph{almost finite} groupoids~\cite[Theorem~7.5]{MatHom} and for \emph{purely infinite} groupoids~\cite[Theorem~5.2]{MatTFG}. Just as with SFT-groupoids, we will see that the more general graph groupoids studied here also belong to the latter class.

\begin{definition}[{\cite[Definition~4.9]{MatTFG}}]\label{def:PI}
An effective ample groupoid $\mathcal{G}$ with compact unit space is said to be \emph{purely infinite} if there for every clopen subset $A \subseteq \mathcal{G}^{(0)}$ exists compact bisections $U,V \subseteq \mathcal{G}$ satisfying $s(U) = s(V) = A$ and $r(U) \sqcup r(V) \subseteq A$.
\end{definition}

\begin{proposition}\label{prop:PI}
Let $E$ be a graph satisfying the AH~criteria.  Then the groupoid $\mathcal{G}_E  \vert_Y$ is purely infinite for each clopen $Y \subseteq \partial E$.
\end{proposition}
\begin{proof}
Although the proof of~\cite[Lemma~6.1]{MatTFG} remains valid with minor modifications in the presence of infinite emitters, we give a brief argument in our notation for the convenience of the reader. Since pure infiniteness passes to restrictions it suffices to consider~$Y = \partial E$.

Let $A \subseteq \partial E$ be given. By compactness we can express~$A = \sqcup_{i=1}^m Z(\mu_i \setminus F_i)$ as a finite union of punctured cylinder sets. By the description following Definition~\ref{def:AH}, any vertex lying outside the nontrivial strongly connected component of $E$ is regular. And any path from such a vertex eventually ends up in the nontrivial connected component. This means that by partitioning the cylinder sets  $Z(\mu_i \setminus F_i)$ into superpaths, we may without loss of generality assume that $r(\mu_i)$ lie in the nontrivial connected component for each~$i$. Thus we can, for each~$i$, find two disjoint cycles $\nu_i, \nu_i'$ based at $r(\mu_i)$. Using these we define bisections~$U= \sqcup_{i=1}^m  Z(\mu_i \nu_i, F_i, \mu_i)$ and $V= \sqcup_{i=1}^m Z(\mu_i \nu_i', F_i, \mu_i)$ which we see satisfy the conditions in Definition~\ref{def:PI}.
\end{proof}



\begin{remark}Recently, more general notions of pure infiniteness for étale groupoids have appeared in the works of Suzuki~\cite{Suz} and Ma~\cite{Ma}. However, for ample minimal groupoids with compact unit space, as in the setting of this paper, both notions agree with Matui's. Furthermore, they imply Anantharaman-Delaroche's notion of \emph{locally contracting}~\cite{AD}.  On a somewhat related note, there is also the recent preprint~\cite{ADS} in which the (not necessarily simple) pure infiniteness of graph \mbox{$C^*$-algebras} (of row-finite graphs without sinks) is characterized solely in terms of the graph groupoid, by means of the paradoxicality notion from~\cite{BL20}.
\end{remark}

The inclusion $\ker(I_{\text{ab}}) \subseteq \im(j)$ is intimately related to the kernel of the index map being generated by transpositions, as encapsulated by the following definition.

\begin{definition}[{\cite[Definition~2.11]{MatProd}}]
An effective ample Hausdorff groupoid $\mathcal{G}$ is said to have \emph{Property~TR} if $\mathcal{S}(\mathcal{G}) = \ker(I)$.
\end{definition}

By Proposition~\ref{prop:PI} and~\cite[Theorem~4.4]{MatProd} it suffices to establish Property~TR in order to verify the AH~conjecture for graph groupoids. Therefore, the rest of the paper is mostly devoted to demonstrating that graph groupoids do have Property~TR. 

\begin{remark}\label{rem:TR}
In general, Property~TR implies the inclusion $\ker(I_{\text{ab}}) \subseteq \im(j)$, i.e.~exactness at $\llbracket \mathcal{G} \rrbracket_{\text{ab}}$ in~\eqref{eq:AH}. The converse holds if the commutator subgroup $\DD(\llbracket \mathcal{G} \rrbracket)$ is simple. For then $\DD(\llbracket \mathcal{G} \rrbracket) = \mathcal{A}(\mathcal{G})$, where $\mathcal{A}(\mathcal{G})$ denotes the ``alternating'' subgroup of~$\mathcal{S}(\mathcal{G})$ defined in~\cite{Nek19}. The group $\DD(\llbracket \mathcal{G} \rrbracket)$ is known to be simple for minimal groupoids which are either almost finite or purely infinite~ \cite{MatTFG}. So for these two classes of groupoids we see that Property~TR is in fact equivalent to the AH~conjecture.
\end{remark}




We close this section by observing, as was done in~\cite{MatTFG}, that to establish Property~TR it suffices to only consider elements in the topological full group whose support is a proper subset of the unit space. Although an easy observation, this is needed for the proof of the main result to work.

\begin{lemma}\label{lem:support}
Let $\mathcal{G}$ be an ample effective Hausdorff groupoid. If all elements $\alpha \in \llbracket \mathcal{G} \rrbracket$ which satisfy $I(\alpha) = 0 \in H_1(\mathcal{G})$ and $\supp(\alpha) \neq \mathcal{G}^{(0)}$ are products of transpositions, then~$\mathcal{G}$ has Property~TR.
\end{lemma}
\begin{proof}
Let $\alpha \in \llbracket \mathcal{G} \rrbracket \setminus \{\id\}$ be given and suppose $I(\alpha) = 0 \in H_1(\mathcal{G})$. As $\alpha$ is not the identity, $\supp(\alpha)$ is non-empty. And then there is some compact open set $Z \subseteq \mathcal{G}^{(0)}$ such that $\alpha(Z) \cap Z = \emptyset$ . We define a transposition $\tau \in \mathcal{S}(\mathcal{G})$ by setting~${\tau = \alpha}$  on~$Z$, $\tau = \alpha^{-1}$ on~$\alpha(Z)$ and~$\tau = \id$ elsewhere. Then $\supp(\tau) = \alpha(Z) \sqcup Z$ and ${\supp(\tau  \alpha) \subseteq \mathcal{G}^{(0)} \setminus (\alpha(Z) \sqcup Z) \subsetneq \mathcal{G}^{(0)}}$. Since both $\alpha$ and $\tau$ (being a transposition) are in the kernel of the index map, so is their product, and by assumption $\tau  \alpha$ is then a product of transpositions. But then $\alpha$ is clearly also a product of transpositions.
\end{proof}

\section{Cancellation for AF-groupoids}\label{sec:cancellation}

\emph{Cancellation} for ample Hausdorff groupoids was introduced by Matui in~\cite{MatProd}, and it bears resemblance to the cancellation property (in $K$-theory) for \mbox{$C^*$-algebras} (see~\cite{Rordam}).

\begin{definition}
An ample Hausdorff groupoid $\mathcal{G}$ is said to have \emph{cancellation} if whenever one has $[1_A] = [1_B]$ in $H_0(\mathcal{G})$ for $\emptyset \neq A,B \subseteq \mathcal{G}^{(0)}$ compact open, there exists a bisection~$U \subseteq \mathcal{G}$ with $s(U) = A$ and $r(U) = B$. 
\end{definition}

In order to prove our main result we are going to need the fact that AF-groupoids have cancellation. This might be known to experts, but we were unable to locate a reference. Theorem~6.12 in~\cite{MatHom} covers minimal AF-groupoids with compact unit space, but we need cancellation for the skew product $\mathcal{G}_E \times_{c_E} \ZZ$, which is generally neither minimal nor does it have compact unit space. So we provide a proof here, which we divide into three lemmas in terms of permanence properties of cancellation.

\begin{lemma}
Let $\mathcal{G}$ be an ample Hausdorff groupoid. If $\mathcal{G}_1 \subseteq \mathcal{G}_2 \subseteq \mathcal{G}_3 \subseteq \ldots$ are open subgroupoids of $\mathcal{G}$ with $\cup_{n=1}^\infty \mathcal{G}_n = \mathcal{G}$, and each $\mathcal{G}_n$ has cancellation, then $\mathcal{G}$ has cancellation.
\end{lemma}
\begin{proof}
Let $A,B \subseteq \mathcal{G}^{(0)}$ be compact open and suppose $[1_A] = [1_B]$ in $H_0(\mathcal{G})$. This means that $1_A - 1_B = \delta_1(f)$ for some $f \in C_c(\mathcal{G}, \ZZ)$. As the support of $f$ is compact we must have $\supp(f) \subseteq \mathcal{G}_n$ for some $n \in \NN$. By possibly increasing $n$ we may suppose that~$A,B \subseteq \mathcal{G}_n^{(0)}$ as well. We have $f \vert_{\mathcal{G}_n} \in C_c(\mathcal{G}_n, \ZZ)$ and $\delta_1(f \vert_{\mathcal{G}_n}) = \delta_1(f) = 1_A - 1_B$. Cancellation in $\mathcal{G}_n$ now provides a bisection $U \subseteq \mathcal{G}_n \subseteq \mathcal{G}$ with $s(U) = A$ and $r(U) = B$. 
\end{proof}

\begin{lemma}
If $\mathcal{G}_1$ and $\mathcal{G}_2$ are ample Hausdorff groupoids with cancellation, then the disjoint union groupoid $\mathcal{G}_1 \sqcup \mathcal{G}_2$ has cancellation.
\end{lemma}
\begin{proof}
Let $A,B \subseteq (\mathcal{G}_1 \sqcup \mathcal{G}_2)^{(0)}$ be compact open and suppose that we have  $[1_A] = [1_B]$ in~${H_0(\mathcal{G}) \cong H_0(\mathcal{G}_1) \oplus H_0(\mathcal{G}_2)}$. Let $f \in C_c(\mathcal{G}_1 \sqcup \mathcal{G}_2, \ZZ)$ be such that $\delta_1(f) = 1_A - 1_B$. We can write   $(\mathcal{G}_1 \sqcup \mathcal{G}_2)^{(0)} =  \mathcal{G}_1^{(0)}  \sqcup \mathcal{G}_2^{(0)}$, $A = A_1 \sqcup A_2$, $B = B_1 \sqcup B_2$ and $f = f_1 + f_2$ respecting this decomposition. It is now clear that $\delta_1(f_1) = 1_{A_1} - 1_{B_1}$ and $\delta_1(f_2) = 1_{A_2} - 1_{B_2}$, so by cancellation in $\mathcal{G}_1$ and $\mathcal{G}_2$ we obtain bisections $U_1 \subseteq \mathcal{G}_1$ and $U_2 \subseteq \mathcal{G}_2$ with $s(U_1) = A_1$, $r(U_1) = B_1$, $s(U_2) = A_2$ and $r(U_2) = B_2$. Setting $U = U_1 \sqcup U_2$ does the trick.
\end{proof}

\begin{lemma}\label{lem:canc}
Let $X$ be a zero-dimensional compact Hausdorff space and let $n \in \mathbb{N}$. Then the elementary groupoid of type $n$, $X \times \mathcal{R}_n$, has cancellation.
\end{lemma}
\begin{proof}
Denote $\mathcal{K} \coloneqq X \times \mathcal{R}_n$, and write $\mathcal{K}^{(0)} = \sqcup_{i=1}^n X_i$, where $X_i = X \times \{i\}$. Then $X_1$ is a full clopen in $\mathcal{K}$ and $\mathcal{K} \vert_{X_1} \cong X$, so we have \[H_0(\mathcal{K}) \cong H_0(\mathcal{K} \vert_{X_1}) \cong H_0(X) = C(X, \ZZ).\]

Suppose that $A,B \subseteq \mathcal{K}^{(0)}$ are clopen subsets with $[1_A] = [1_B]$ in $H_0(\mathcal{K})$. We partition~$A$ by writing $A = \sqcup_{i=1}^n A_i \times \{i\}$ for $A_i \subseteq X$ clopen. Let $B_i$ be similar for $B$. The bisections $A_i \times \{(1,i)\} \subseteq \mathcal{K}$ have source $A_i \times \{i\}$ and range $A_i \times \{1\}$. By~\eqref{eq:0inverse} this means that under the isomorphism $H_0(\mathcal{K}) \cong C(X, \ZZ)$ above, the element $[1_A] \in H_0(\mathcal{K})$ maps to the function $f_A \coloneqq \sum_{i=1}^n 1_{A_i} \in C(X, \ZZ)$, and similarly $[1_B] \mapsto f_B$.


Since $f_A = f_B$ and they are both sums of indicator functions we can find $m_j \in \mathbb{N}$ and $C_j \subseteq X$ clopen such that ${f_A = f_B = \sum_{j=1}^J m_j 1_{C_j}}$. We can think of $f_A$ (and $f_B$) being produced by taking each of the parts $A_i$ and ``projecting'' them down and then stacking them on top of eachother. The height at a point becomes the function value of $f_A$. For each $C_j$ we have that $C_j \times \{i\} \subseteq A_i$ for precisely $m_j$ indices $i$, and we have the same for the $B_i$'s. For fixed $j$ denote these indices for $A$ by $i_1, \ldots, i_{m_j}$, and denote them by $i'_1, \ldots, i'_{m_j}$ for $B$. Then define a bisection $U_j \subseteq \mathcal{K}$ by $U_j = 
 \sqcup_{k=1}^{m_j} U_k$, where~$U_k \coloneqq C_{j} \times (i'_k, i_k)$. Finally setting $U = \sqcup_{j=1}^J U_j$ gives a bisection with~$s(U) = A$ and~$r(U) = B$.
\end{proof}

\begin{theorem}\label{thm:cancellation}
Any AF-groupoid has cancellation.
\end{theorem}
\begin{proof}
Let $\mathcal{G}$ be an AF-groupoid. Then we can write $\mathcal{G} = \cup_{n=1}^\infty \mathcal{G}_n$ as an increasing union of open elementary ample subgroupoids. By~\cite[Lemma~3.4]{GPS04} each subgroupoid decomposes as \[\mathcal{G}_n \cong \left( \sqcup_{i=1}^{I_n} X_{i,n} \times \mathcal{R}_{m_{i,n}} \right) \bigsqcup Y_n, \]
where each $X_{i,n}$ is a zero-dimensional compact Hausdorff space, and where $Y_n$ is empty if~$\mathcal{G}^{(0)}$ is compact and zero-dimensional, locally compact non-compact and Hausdorff if~$\mathcal{G}^{(0)}$ is non-compact. Since the trivial groupoid $Y_n$ clearly has cancellation, the result follows by combining the three lemmas above. 
\end{proof}

We end this section by observing that in an AF-groupoid, a non-empty subset of the unit space always gives rise to a nonzero element in homology. This is not so for all groupoids with cancellation (e.g.\  the SFT-groupoid of the full $2$-shift, $\mathcal{G}_{[2]}$).

\begin{corollary}\label{cor:zeroCancel}
Let $\mathcal{G}$ be an AF-groupoid. If $A \subseteq \mathcal{G}^{(0)}$ is compact open, then $[1_A] = 0$ in $H_0(\mathcal{G})$ if and only if $A = \emptyset$.
\end{corollary}
\begin{proof}
Follows from the proofs above by considering $B = \emptyset$, i.e.\ $1_B = 0$.
\end{proof}

\section{Two long exact sequences in homology}\label{sec:LES}

Let us first describe a long exact sequence in homology coming from a cocycle. Let $\mathcal{G}$ be an ample Hausdorff groupoid with a cocycle $c \colon \mathcal{G} \to \ZZ$. Let $\pi$ denote the canonical projection from~$\G\times_c \ZZ$ onto $\G$, i.e.\ $\pi(g,m) = g$. Also, let ${\rho \coloneqq \widehat{c}_1 \colon \G\times_c \ZZ \to \G\times_c \ZZ}$, i.e.\ $\rho(g,m) = (g,m+1)$. Since these are étale homomorphisms, they induce chain maps~${\pi_\bullet \colon C_\bullet(\G\times_c \ZZ, \ZZ) \to C_\bullet(\G, \ZZ)}$ and $\rho_\bullet \colon C_\bullet(\G\times_c \ZZ, \ZZ) \to C_\bullet(\G \times_c \ZZ, \ZZ)$ on the chain complexes that define the homology groups.
In fact, $\id - \rho_\bullet$ and $\pi_\bullet$ form a short exact sequence of complexes, which in turn induces a long exact sequence in homology.

\begin{proposition}[{\cite[Lemma~1.4]{Ort}}]\label{prop:SES1}
Let $\mathcal{G}$ be an ample Hausdorff groupoid and let~${c \colon \mathcal{G} \to \ZZ}$ be a cocycle. Then there is a long exact sequence 
\[ \begin{tikzcd}[column sep = 3em]
\cdots \arrow{r}{H_1(\pi_\bullet)} &     H_1(\mathcal{G}) \arrow{r}{\partial_1} & H_0(\G\times_c \ZZ) \arrow{r}{\id - H_0(\rho_\bullet)} &   H_0(\G\times_c \ZZ) \arrow{r}{H_0(\pi_\bullet)} & H_0(\mathcal{G})  \arrow{r} & 0,
\end{tikzcd} \]
where ${\partial_n}$ denotes the connecting homomorphism.
\end{proposition}

The maps on the zeroth level are given by 
\[ H_0(\rho_\bullet)\left(\left[1_{A \times \{i\}}\right]\right) = \left[1_{A \times \{i+1\}}\right]
\quad	\text{and}	   \quad 
H_0(\pi_\bullet)\left(\left[1_{A \times \{i\}}\right]\right) = \left[1_A\right]  \]
for $A \subseteq \mathcal{G}^{(0)}$ compact open and $i \in \ZZ$. In the case of graph groupoids, we will see later that the first connecting homomorphism $\partial_1 \colon H_1(\mathcal{G}) \to H_0(\G\times_{c} \ZZ)$ can be described explicitly, and that this will allow us to describe the image of the index map. In order to do that, we are going to need a particular part of the proof of~\cite[Lemma~1.4]{Ort} pertaining lifts by $\id - \rho_0$. We record this lifting in Lemma~\ref{lem:SES1} below, whose proof itself is an easy calculation.

\begin{lemma}\label{lem:SES1}
Let $c \colon \mathcal{G} \to \ZZ$ be a cocycle on an ample Hausdorff groupoid $\mathcal{G}$. Then for any $A \subseteq \mathcal{G}^{(0)}$ compact open and $k \in \ZZ$ we have 
\[1_{A \times \{k\}} - 1_{A \times \{0\}}  = \begin{cases}
(\id - \rho_0)\left(  - \sum_{i=0}^{k-1} 1_{A \times \{i\}}    \right) & k > 0, \\ 
0    & k = 0, \\
(\id - \rho_0)\left( \sum_{i= k}^{-1} 1_{A \times \{i\}}    \right)    & k < 0.
\end{cases} \]
\end{lemma}

The next long exact sequence in homology arises from open invariant subsets of the unit space. This is akin to the six-term exact sequences arising from nested ideals in filtered $K$-theory of \mbox{$C^*$-algebras}, as in e.g.~\cite{Rest}. Let $\mathcal{G}$ be an ample Hausdorff groupoid and let $Z \subseteq Y \subseteq \mathcal{G}^{(0)}$ be open sets. The inclusion $\iota \colon \mathcal{G} \vert_Z \hookrightarrow \mathcal{G} \vert_Y$ induces the chain map $\iota_n \colon C_c( (\mathcal{G} \vert_Z)^{(n)}, \ZZ) \to C_c((\mathcal{G} \vert_Y)^{(n)}, \ZZ)$ which is given by extending functions to be~$0$ outside~$\mathcal{G} \vert_Z$. Let $\kappa _n \colon C_c((\mathcal{G} \vert_Y)^{(n)}, \ZZ) \to C_c((\mathcal{G} \vert_{(Y \setminus Z)})^{(n)}, \ZZ)$ denote the restriction maps. Taking such restrictions commute with the differentials $\delta_n$, so $\kappa_\bullet$ is also a chain map. If the sets $Z$ and $Y$ are both $\mathcal{G}$-invariant, then~$\iota_\bullet$ and~$\kappa_\bullet$ form a short exact sequence of complexes, which appear in an unpublished manuscript by T.~M.~Carlsen and the second author. This then results in the following long exact sequence in homology.

\begin{proposition}\label{prop:SES}
Let $\mathcal{G}$ be an ample Hausdorff groupoid and assume that $Z \subseteq Y \subseteq \mathcal{G}^{(0)}$ are open and $\mathcal{G}$-invariant. Then there is a long exact sequence
\[ \begin{tikzcd}
\cdots \arrow{r}{H_1(\kappa_\bullet)} &     H_1\left(\mathcal{G} \vert_{(Y \setminus Z)}\right) \arrow{r} & H_0\left(\mathcal{G} \vert_Z\right) \arrow{r}{H_0(\iota_\bullet)} &   H_0\left(\mathcal{G} \vert_Y\right) \arrow{r}{H_0(\kappa_\bullet)} & H_0\left(\mathcal{G} \vert_{(Y \setminus Z)}\right)  \arrow{r} & 0.
\end{tikzcd} \]
\end{proposition}

\section{The homology groups of a graph groupoid}\label{sec:graphhom}

We have already seen that the homology groups of a graph groupoid coincide with the~$K$-groups of its groupoid \mbox{$C^*$-algebra}. We will make use of this in the final section. However, in order to prove Property~TR for the graph groupoid $\mathcal{G}_E$ we are going to relate the first homology group $H_1(\mathcal{G}_E)$ to the homology groups $H_0(\G_E\times_{c_E} \ZZ)$ and $H_0(\mathcal{H}_E)$. In this section we will use the long exact sequences from the previous section to deduce the following embeddings:
\[ H_1(\mathcal{G}_E)\hookrightarrow H_0(\mathcal{H}_E) \hookrightarrow H_0(\G_E\times_{c_E} \ZZ). \]
This will be done in three steps: first we show that $H_1(\mathcal{G}_E) \hookrightarrow H_0(\G_E\times_{c_E} \ZZ)$, then that~$H_0(\mathcal{H}_E)   \hookrightarrow H_0(\G_E\times_{c_E} \ZZ)$ and finally that $H_1(\mathcal{G}_E)\hookrightarrow H_0(\mathcal{H}_E)$. The reason we need three steps (and not two) is that the third embedding relies on the first two.

\subsection{The first embedding} 
Let us begin by describing the zeroth homology group of the skew product $\G_E\times_{c_E} \ZZ$. Recall that $(\G_E\times_c \ZZ)^{(0)}$ is identified with $\partial E \times \ZZ$. Observe that we have\footnote{By \emph{$\spann$} we mean linear combinations over $\ZZ$.}
\begin{align*}
H_0(\G_E\times_c \ZZ) & =\spann\{\left[1_A \right] \mid A \subseteq \partial E \times \ZZ \text{ compact open} \} \\
& =\spann\{\left[1_{Z(\mu \setminus F) \times \{i\}} \right] \mid \mu\in E^*,\,F\subseteq_\text{finite} r(\mu)E^1,\, i \in \ZZ \} \\ 
& =\spann\{\left[1_{Z(\mu) \times \{i\}} \right] \mid \mu\in E^*,\, i \in \ZZ \},
\end{align*}
since $1_{Z(\mu\setminus F) \times \{i\}} = 1_{Z(\mu) \times \{i\}} - \sum_{e\in F} 1_{Z(\mu e) \times \{i\}}$. These spanning elements satisfy the following relations in~$H_0(\G_E\times_{c_E} \ZZ)$:
\begin{align}
\left[ 1_{Z(\mu) \times \{i\}} \right] &=\left[ 1_{Z(\sigma_E(\mu)) \times \{i+1\}}\right] \text{ if } \vert \mu \vert \geq 1, \label{al:rel1} \\
\left[ 1_{Z(\mu) \times \{i\}} \right] &=\left[ 1_{Z(e \mu ) \times \{i-1\}}\right] \text{ for any } e \in E^1s(\mu), \\
\left[ 1_{Z(\mu) \times \{i\}} \right] &= \sum_{e \in r(\mu)E^1}  \left[ 1_{Z(\mu e) \times \{i \}}\right] \text{ if } r(\mu) \text{ is a regular vertex}, \\
\left[ 1_{Z(\mu) \times \{i\}} \right] &=\left[ 1_{Z(\nu) \times \{i\}} \right] \text{ if } \vert \mu \vert = \vert \nu \vert \text{ and } r(\mu) = r(\nu).
\end{align}
For all of the sets appearing in the indicator functions above it is easy to find a bisection in $\G_E\times_c \ZZ$ whose source is the left hand side and whose range is the right hand side. From repeated use of the relation~\eqref{al:rel1} we see that we can even write 
\[H_0(\G_E\times_{c_E} \ZZ) = \spann\{\left[ 1_{Z(v) \times \{i\}} \right] \mid v \in E^0,\, i \in \ZZ \}, \]
since $\left[1_{Z(\mu) \times \{i\}} \right] = \left[1_{Z(r(\mu)) \times \{i + \vert \mu \vert \}} \right]$.

Let us now consider the long exact sequence in homology that we get from the canonical cocycle $c_E$ on a graph groupoid $\mathcal{G}_E$. Since $\G_E\times_{c_E} \ZZ$ is an AF-groupoid (Corollary~\ref{cor:AF}), its $H_1$ group vanishes, and therefore the first part of the long exact sequence from Proposition~\ref{prop:SES1} becomes
\begin{equation}\label{eq:4exact}
\begin{tikzcd}[column sep = 2.7em]
0 \arrow{r} &[-20pt]      H_1(\mathcal{G}_E) \arrow{r}{\partial_1} &[-5pt] H_0(\G_E\times_{c_E} \ZZ) \arrow{r}{\id - H_0(\rho_\bullet)} &   H_0(\G_E\times_{c_E} \ZZ) \arrow{r}{H_0(\pi_\bullet)} & H_0(\mathcal{G}_E)  \arrow{r} &[-20pt]  0.
\end{tikzcd} 
\end{equation}
The map $H_0(\rho_\bullet) \colon H_0(\G_E\times_{c_E} \ZZ) \to H_0(\G_E\times_{c_E} \ZZ)$ is given by 
\[ H_0(\rho_\bullet)\left(\left[ 1_{Z(v) \times \{i\}}  \right]\right) = \left[1_{Z(v) \times \{i+1\}} \right]  \]
for $v \in E^0$ and $i \in \ZZ$. The connecting homomorphism $\partial_1$ will be described explicitly in the proof of Lemma~\ref{lem:indeximage}. From the exactness of~\eqref{eq:4exact} we deduce the following.

\begin{proposition}\label{prop:H01}	
Let $E$ be a graph and let $H_0(\rho_\bullet) \colon H_0(\G_E\times_{c_E} \ZZ) \to H_0(\G_E\times_{c_E} \ZZ)$ be as above. Then 
\[ H_0(\mathcal{G}_E)  \cong \coker(\id - H_0(\rho_\bullet))
 \qquad \text{and} \qquad 
 H_1(\mathcal{G}_E)  \cong \ker(\id - H_0(\rho_\bullet)).   \]
\end{proposition}

\begin{remark}\label{rem:fullness} 
In the proof of~\cite[Theorem~4.14]{MatHom}, Matui obtained formulas similar to those in Proposition~\ref{prop:H01} using a spectral sequence. This relied on the fact that $H_0(\mathcal{H}_E)$ and~$H_0(\G_E\times_{c_E} \ZZ)$ can be identified when $E$ is finite (or more generally row-finite) with no sinks. For then $\partial E \times \{0\}$ is \mbox{$(\G_E \times_{c_E} \ZZ)$-full}, so $\mathcal{H}_E$ is Kakutani equivalent to $\G_E \times_{c_E} \ZZ$.
This allowed Matui to immediately realize $H_1(\mathcal{G}_E)$ as a subgroup of~$H_0(\mathcal{H}_E)$.

At this point we encounter a significant difference from the finite graph case. For when $E$ has singular vertices one can show that $\partial E \times \{0\}$ \emph{never} is~$(\G_E \times_{c_E} \ZZ)$-full. So in our setting we cannot necessarily identify $H_0(\mathcal{H}_E)$ with~$H_0(\G_E \times_{c_E} \ZZ)$. We will, however, be able to identify the former with a subgroup of the latter.
\end{remark}

\subsection{The second embedding} 


Recall that $\mathcal{H}_E = \ker(c_E) \subseteq \mathcal{G}_E$ and from
Lemma~\ref{lem:kerc} we have that $\mathcal{H}_E \cong \left( \G_E \times_{c_E} \ZZ \right) \vert_{\partial E \times \{0\}}$  via the identification $(x,0,y) \leftrightarrow ((x,0,y),0)$. In~$H_0(\mathcal{H}_E)$ we have the relation
\[\left[ 1_{Z(\mu)} \right] = \left[1_{Z(\nu)} \right] \]
whenever $\mu, \nu \in E^*$ satisfy $\vert \mu \vert = \vert \nu \vert$ and $r(\mu) = r(\nu)$. The element ${\left[ 1_{Z(\mu)} \right] \in H_0(\mathcal{H}_E)}$ corresponds to~$\left[ 1_{Z(\mu) \times \{0\}} \right] \in H_0(\left( \G_E \times_{c_E} \ZZ \right) \vert_{\partial E \times \{0\}})$ under the identification above. On the other hand, the indicator function $1_{Z(\mu) \times \{0\}}$ gives rise to an element $\left[ 1_{Z(\mu) \times \{0\}} \right]$ in~$H_0(\G_E\times_{c_E} \ZZ)$ as well. A priori, these are different, but we will see that mapping~${\left[ 1_{Z(\mu)} \right] \in H_0(\mathcal{H}_E)}$ to~$\left[ 1_{Z(\mu) \times \{0\}} \right] \in H_0(\G_E\times_{c_E} \ZZ)$ actually gives an embedding of groups. So that in the end, there is no ambiguity. The map from $H_0(\mathcal{H}_E)$ to $H_0(\G_E\times_{c_E} \ZZ)$ proposed above extends to arbitrary elements by 
\[H_0(\mathcal{H}_E) \ni [f] \longmapsto [ f \times 0] \in H_0(\G_E \times_{c_E} \ZZ)   \] 
for $f \in C_c(\partial E, \ZZ)$, where  $f \times 0 \in C_c(\partial E  \times \ZZ, \ZZ)$ is given by 
\[ (f \times 0)(x,m) = \begin{cases} f(x) & \text{ if } m = 0, \\
0 & \text{ otherwise.}
\end{cases}  \]
By noting that $\left( \G_E \times_{c_E} \ZZ \right) \vert_{\partial E \times \{0\}} = \mathcal{H}_E \times \{0\} \subseteq \G_E \times \ZZ = \G_E \times_{c_E} \ZZ$ as sets, it is not hard to see that this is a well-defined homomorphism. Its injectivity will be deduced using the second long exact sequence from Section~\ref{sec:LES}.

\begin{lemma}\label{lem:H0inH0}
Let $E$ be a graph. The homomorphism $\phi \colon H_0(\mathcal{H}_E) \to H_0(\G_E\times_{c_E} \ZZ)$ given by~$\phi([f]) = [ f \times 0]$ for $f \in C_c(\partial E, \ZZ)$ is injective.
\end{lemma}
\begin{proof}
In the setting of Proposition~\ref{prop:SES}, set $\mathcal{G} = \mathcal{G}_E \times_{c_E} \ZZ$, $Y = \mathcal{G}^{(0)} = \partial E \times \ZZ$ and~${X = \partial E \times \{0\}}$. The clopen set $X$ is neither $\mathcal{G}$-full nor invariant, so we instead consider its saturation, namely $Z \coloneqq r(s^{-1}(X))$. In words $Z$ is the smallest $\mathcal{G}$-invariant subset containing $X$. By étaleness, $Z$ is open in $\partial E \times \ZZ$. By its very definition, $X$ is clopen in $Z$ and $\mathcal{G} \vert_Z$-full, hence $\mathcal{H}_E \cong \mathcal{G} \vert_X = \left(\mathcal{G} \vert_Z \right)  \vert_X$ is Kakutani equivalent to~$\mathcal{G} \vert_Z$. The induced isomorphism $H_0(\mathcal{H}_E) \cong H_0(\mathcal{G} \vert_Z)$ maps $[1_{Z(\mu)}]$ to $[1_{Z(\mu) \times \{0\}}]$, where we now consider $1_{Z(\mu) \times \{0\}} \in C_c(Z, \ZZ)$. Since $\mathcal{G}$ is an AF-groupoid and the set $Y \setminus Z$ is closed in~$\mathcal{G}^{(0)}$, the restriction $\mathcal{G} \vert_{(Y \setminus Z)}$ becomes an AF-groupoid (in the relative topology) as well. Its~$H_1$ group then vanishes and the first part of the long exact sequence in Proposition~\ref{prop:SES} becomes
\[ \begin{tikzcd}[column sep = 1.8em]
0 \arrow{r} & H_0\left(\left( \G_E \times_{c_E} \ZZ \right) \vert_Z\right) \arrow{r}{H_0(\iota_\bullet)} &   H_0(\G_E \times_{c_E} \ZZ) \arrow{r}{H_0(\kappa_\bullet)} & H_0\left(\left( \G_E \times_{c_E} \ZZ \right) \vert_{(\partial E \times \ZZ) \setminus Z}  \right)  \arrow{r} & 0.
\end{tikzcd} \]
The map $H_0(\iota_\bullet)$ is given by inclusion (i.e.\ by extending to $0$). So if we compose $H_0(\iota_\bullet)$ with the isomorphism $H_0(\mathcal{H}_E) \cong H_0(\mathcal{G} \vert_Z) = H_0\left(\left( \G_E \times_{c_E} \ZZ \right) \vert_Z\right)$ from above we get $\phi$
back. Its injectivity then follows from the injectivity of $H_0(\iota_\bullet)$.
\end{proof}

\begin{remark}
We can actually describe the set $Z$ from the proof of Lemma~\ref{lem:H0inH0}  explicitly, assuming that $E$ is strongly connected, as follows: 
\[ Z = \{(x,k) \mid x \in E^\infty, k \in \ZZ  \} \bigsqcup  \{(\mu, l) \mid \mu \in \partial E \cap E^*, \ l \geq - \vert \mu \vert  \} \subseteq \partial E \times \ZZ = Y.  \]
The complement is therefore 
\[ Y \setminus Z = (\partial E \times \ZZ) \setminus Z =  \{(\mu, l) \mid \mu \in \partial E \cap E^*, \ l < - \vert \mu \vert  \}.    \]
If $E$ has a singular vertex, then $Z$ is an open and dense proper subset of $\partial E \times \ZZ$, as well as $\mathcal{G}_E \times_{c_E} \ZZ$-invariant. And the complement is non-empty, closed, has empty interior and is also invariant. 
\end{remark} 

\subsection{The third embedding}

From now on we will freely identify $H_0(\mathcal{H}_E)$ with the subgroup generated by the elements $\left[1_{Z(\mu) \times \{0\}} \right]$ for $\mu \in E^*$ inside $H_0(\G_E\times_{c_E} \ZZ)$. The first thing we shall note is that this copy of $H_0(\mathcal{H}_E)$ inside $H_0(\G_E\times_{c_E} \ZZ)$ is invariant under~$H_0(\rho_\bullet)$, provided that $E$ has no sources. Indeed, for~$\mu \in E^*$ 
\[H_0(\rho_\bullet)\left(\left[1_{Z(\mu) \times \{0\}} \right]\right) = \left[1_{Z(\mu) \times \{1\}} \right] = \left[1_{Z(e \mu) \times \{0\}} \right], \]
where $e$ is any edge whose range is $s(\mu)$ (and the equivalence class does not depend on which edge $e$ is chosen). The restriction of $H_0(\rho_\bullet)$ to $H_0(\mathcal{H}_E)$ will be important in the sequel, so we give it a name of its own.

\begin{definition}\label{def:delta}
Let $E$ be an essential graph. By viewing $H_0(\mathcal{H}_E)$ as a subgroup of~${H_0(\G_E\times_{c_E} \ZZ)}$ we define an endomorphism~$\varphi \colon H_0(\mathcal{H}_E) \to H_0(\mathcal{H}_E)$ by \[\varphi\left(\left[1_{Z(\mu) \times \{0\}} \right] \right) = H_0(\rho_\bullet)\left(\left[1_{Z(\mu) \times \{0\}} \right]\right) = \left[1_{Z(e \mu) \times \{0\}} \right], \]
where $e \in E^1s(\mu)$ is arbitrary.
\end{definition} 

In the next section we will see that the image of an element of the topological full group under the index map can be described in terms of the map $\varphi$.

\begin{remark}\label{rem:MatuiDelta}
On page 56 of~\cite{MatTFG} Matui implicitly defines, for any finite strongly connected graph $E$, an automorphism denoted~$\delta$ of $H_0(\mathcal{H}_E)$. Explicitly, $\delta$ is given by
\begin{align*}
  \delta \left(  \left[1_{Z(\mu) \times \{0\}}  \right] \right) & = \left[1_{Z(\sigma_E(\mu)) \times \{0\}} \right] = \left[1_{Z(\mu) \times \{-1\}} \right] \\
\text{for }  \left[1_{Z(\mu) \times \{0\}} \right] & \in H_0(\mathcal{H}_E) = \spann \left\{\left[1_{Z(\mu) \times \{0\}} \right] \mid \mu \in E^{\geq 1} \right\}.
\end{align*}
Hence the homomorphism $\varphi$ from Definition~\ref{def:delta} equals $\delta^{-1}$. But if the graph $E$ has singular vertices, then $\delta$ is no longer globally defined on $H_0(\mathcal{H}_E)$. To see this, note that~$\varphi$ is generally not surjective. 
For example, the elements $\left[1_{Z(w) \times \{0\}} \right]$, where $w$ is an infinite emitter, will generally not be in the image of $\varphi$.
\end{remark}


We are now ready to prove the  the third and final embedding of the homology groups.

\begin{lemma}\label{lem:kerker}
Let $E$ be an essential graph. Then $\ker(\id - H_0(\rho_\bullet)) = \ker(\id - \varphi)$ as subsets of $H_0(\G_E\times_{c_E} \ZZ)$.
\end{lemma}
\begin{proof}
With $\phi$ as in Lemma~\ref{lem:H0inH0} we have the commutative diagram 
\[\begin{tikzcd}[column sep = large, row sep = large]
H_0(\G_E\times_{c_E} \ZZ) \arrow{r}{\id - H_0(\rho_\bullet)} & H_0(\G_E\times_{c_E} \ZZ) \\
H_0(\mathcal{H}_E) \arrow{r}{\id - \varphi} \arrow[hook]{u}{\phi} & H_0(\mathcal{H}_E) \arrow[hook]{u}{\phi}
\end{tikzcd}\]
under which we identify $H_0(\mathcal{H}_E)$ with $\phi(H_0(\mathcal{H}_E)) \subseteq H_0(\G_E\times_{c_E} \ZZ)$. From this it is clear that~${\ker(\id - \varphi)  \subseteq  \ker(\id - H_0(\rho_\bullet))}.$

To prove the reverse inclusion we first show that any element of $H_0(\G_E\times_{c_E} \ZZ)$ can be put in a certain ``standard form''. Each element $\omega \in H_0(\G_E\times_{c_E} \ZZ)$ can be written as 
\[\omega = \sum_{i=-n}^n \sum_{j=1}^{k_i} \lambda_{i,j} \left[1_{Z(v_{i,j}) \times \{i\}} \right], \]
where $\lambda_{i,j}$ are integers and $v_{i,j} \in E^0$. When $i \geq 0$ we have
\begin{equation}\label{eq:pos}
\left[1_{Z(v) \times \{i\}} \right] = \left[1_{Z(\mu) \times \{0\}} \right],
\end{equation}
where $\mu$ is any path of length $i$ in $E$ which ends in $v$. When $v$ is a regular vertex we have 
\begin{equation}\label{eq:reg}
\left[1_{Z(v) \times \{i\}} \right] = \sum_{e \in vE^1} \left[1_{Z(r(e)) \times \{i+1\}} \right].
\end{equation}
So when $i < 0$ we can, by repeated use of~\eqref{eq:reg}, write 
\begin{equation}\label{eq:neg}
\left[1_{Z(v) \times \{i\}} \right] = \sum_{j = i}^{-1} \sum_{k=1}^{K_j} \left[1_{Z(w_{j, k}) \times \{j\}} \right] + \sum_{k=1}^{K_0} \left[1_{Z(v_k) \times \{0\}} \right],
\end{equation}
where each $w_{j, k}$ is an infinite emitter. Combining~\eqref{eq:pos} and~\eqref{eq:neg} we see that we can write the arbitrary element~$\omega$ as
\begin{equation*}
\omega = \sum_{i = -n}^{-1} \sum_{j=1}^{J_i} \lambda_{i,j} \left[1_{Z(w_{i, j}) \times \{i\}} \right] + \sum_{j=1}^{J_0} \lambda_{0,j} \left[1_{Z(\mu_j) \times \{0\}} \right],
\end{equation*}
where $n \in \mathbb{N}$, $\lambda_{i,j} \in \mathbb{Z}$, each $w_{i, j}$ is an infinite emitter and $\mu_j \in E^*$. We may assume that all the $w_{i,j}$'s are different for each fixed $i$.

Suppose now that $\omega \in \ker(\id - H_0(\rho_\bullet))$. We need to show that $\omega \in H_0(\mathcal{H}_E)$ (viewed as a subgroup of $H_0(\G_E\times_{c_E} \ZZ)$). We compute 
\begin{align*}
H_0(\rho_\bullet)(\omega) &= \sum_{i = -n}^{-1} \sum_{j=1}^{J_i} \lambda_{i,j} \left[1_{Z(w_{i, j}) \times \{i + 1\}} \right] + \sum_{j=1}^{J_0} \lambda_{0,j} \left[1_{Z(\mu_j) \times \{1\}} \right] \\ 
& = \sum_{i = -n+1}^{0} \sum_{j=1}^{J_{i-1}} \lambda_{i-1,j} \left[1_{Z(w_{i-1, j}) \times \{i\}} \right] + \sum_{j=1}^{J_0} \lambda_{0,j} \left[1_{Z(e_j \mu_j) \times \{0\}} \right],
\end{align*}
where $e_j$ is any edge ending in $s(\mu_j)$. From this we get 
\begin{align}\label{eq:1.4}
0 &= \omega - H_0(\rho_\bullet)(\omega) =  \sum_{j=1}^{J_{-n}} \lambda_{-n,j} \left[1_{Z(w_{-n, j}) \times \{-n\}} \right] \nonumber \\
& + \sum_{i = -n+1}^{-1}  \left( \sum_{j=1}^{J_i} \lambda_{i,j} \left[1_{Z(w_{i, j}) \times \{i\}} \right]   -  \sum_{j=1}^{J_{i-1}} \lambda_{i-1,j} \left[1_{Z(w_{i-1, j}) \times \{i\}} \right]     \right)     \nonumber \\
& + \sum_{j=1}^{J_0} \left(      \lambda_{0,j} \left[1_{Z(\mu_j) \times \{0\}} \right]    -   \lambda_{0,j} \left[1_{Z(e_j \mu_j) \times \{0\}} \right]     \right)    - \sum_{j=1}^{J_{-1}} \lambda_{-1,j} \left[1_{Z(w_{-1, j}) \times \{0\}} \right].
\end{align}
As $w_{-n, j}$ is singular, each of~$\left[1_{Z(w_{-n, j}) \times \{-n\}} \right]$ generates a free summand of ${H_0(\G_E\times_{c_E} \ZZ)}$ by Lemma~\ref{lem:freesummand}. Since all the other terms have  a strictly smaller second coordinate, in order for the right hand side of~\eqref{eq:1.4} to be $0$ we must have $\lambda_{-n,j} = 0$ for all $1 \leq j \leq J_{-n}$. Thus we may replace $-n$ with $-n+1$ in the expression for $\omega$. Arguing inductively we get that $\lambda_{i,j} = 0$ for all $-1 \leq i \leq -n$ and~$1 \leq j \leq J_i$. Hence the expression for $\omega$ reduces to 
\[\omega = \sum_{j=1}^{J_0} \lambda_{0,j} \left[1_{Z(\mu_j) \times \{0\}} \right],  \]
from which we see that $\omega \in H_0(\mathcal{H}_E)$.
\end{proof}

\section{The image of the index map}\label{sec:indexmap}

Recall the index map $I \colon \llbracket \mathcal{G}_E \rrbracket \to H_1(\G_E)$ described in Section~\ref{sec:AH}. Our main goal is to establish that the kernel of the index map is generated by transpositions (i.e.\ property~TR) for minimal graph groupoids. To that end, the goal of this section is to describe the image $I(\alpha) \in H_1(\mathcal{G}_E)$ of an element $\alpha \in \llbracket \mathcal{G}_E \rrbracket$ under the identification $H_1(\mathcal{G}_E)  \cong \ker(\id - \varphi)$ from Proposition~\ref{prop:H01} and Lemma~\ref{lem:kerker}. 

\subsection{Graded partitions}

The identification desribed above will be done in terms of the following ``graded partitions'' as defined in~\cite[page~60]{MatTFG}. 

\begin{definition}\label{def:gradPart}
Let $E$ be a graph. For $\alpha = \pi_U \in \llbracket \mathcal{G}_E \rrbracket$ and $k \in \ZZ$ we define the set 
\[S_{\alpha }(k) \coloneqq s\left(U \cap c_E^{-1}(k) \right) = \{x \in \partial E \mid (\alpha(x), k, x) \in U  \}. \] 
\end{definition}

Note that each $S_{\alpha }(k)$ is clopen and that $\partial E \setminus \supp(\alpha) \subseteq S_\alpha(0)$, i.e.\ $S_\alpha(0)$ contains the largest (cl)open set fixed by $\alpha$. As $\supp(\alpha)$ is compact, $S_{\alpha }(k)$ is also compact when~$k \neq 0$. This implies that only finitely many $S_{\alpha }(k)$'s will be non-empty. Hence these form a finite partition of the boundary path space $\partial E$. We make a few more observations about these graded partitions that we are going to need in the proof of the main result.

\begin{lemma}\label{lem:alfaSalfa}
Let $E$ be a graph and let $\alpha \in \llbracket \mathcal{G}_E \rrbracket$. We have ${\alpha(S_\alpha(k)) = S_{\alpha^{-1}}(-k)}$ for each $k \in \ZZ$.
\end{lemma}
\begin{proof}
Recall that $U_\alpha$ denotes the unique bisection which satisfies~$\alpha = \pi_{U_\alpha}$. Suppose that ${x \in S_\alpha(k)}$, i.e.~$(\alpha(x), k, x) \in U_\alpha$. Then $(x, - k, \alpha(x)) \in (U_\alpha)^{-1} = U_{\alpha^{-1}}$. This shows that $\alpha(x) \in S_{\alpha^{-1}}(-k)$, hence we have the containment $\alpha(S_\alpha(k)) \subseteq S_{\alpha^{-1}}(-k)$ for all integers $k$. Since these sets form partitions of the unit space we must necessarily have equality.
\end{proof}

The next observation is that when two elements of the topological full group have the same graded partitions, then their difference belongs to the AF-kernel of the cocycle. 

\begin{lemma}\label{lem:betaalphainverse}
Let $E$ be a graph and let $Y \subseteq \mathcal{G}_E^{(0)} = \partial E$ be clopen. Suppose $\alpha, \beta  \in \llbracket  \mathcal{G}_E \vert_Y \rrbracket$ satisfy $S_\alpha(k) = S_\beta(k)$ for all $k \in \ZZ$. Then $\beta \alpha^{-1}  \in \llbracket  \mathcal{H}_E \vert_Y \rrbracket$, that is, $U_{\beta \alpha^{-1}} \subseteq c_E^{-1}(0)$.
\end{lemma}
\begin{proof}
We claim that because the graded partitions of $\alpha$ and $\beta$ are the same, we must have
\[S_{\beta \alpha^{-1}}(k) = \begin{cases}
Y & k = 0, \\
\emptyset & k \neq 0.
\end{cases} \]
And once we have this we immediately see that each element $g = (x,k,y) \in U_{\beta \alpha^{-1}}$ must have $k=0$, i.e.\ that $U_{\beta \alpha^{-1}} \subseteq c_E^{-1}(0)$.

To prove the claim, take an arbitrary point $y \in Y$. Then $y \in S_{\alpha^{-1}}(k)$ for some $k$. By Lemma~\ref{lem:alfaSalfa} we have $\alpha^{-1}(y) \in S_\alpha(-k) = S_\beta(-k)$. And then $g = (\alpha^{-1}(y), k, y) \in U_{\alpha^{-1}}$ and $h = (\beta \alpha^{-1}(y), -k, \alpha^{-1}(y)) \in U_\beta$. From this we get $h \cdot g = (\beta \alpha^{-1}(y) ,0 ,y)  \in U_{\beta \alpha^{-1}}$, hence $y \in S_{\beta \alpha^{-1}}(0)$, which proves the claim.
\end{proof}

The third lemma describes what happens to the graded partition of an element of the topological full group when we perturb it with a particular transposition.

\begin{lemma}\label{lem:Sbeta}
Let $E$ be a graph and let $Y \subseteq \mathcal{G}_E^{(0)} = \partial E$ be clopen. Let $V \subseteq \mathcal{G}_E \vert_Y$ be a compact bisection with disjoint source and range, and such that $V \subseteq c_E^{-1}(K)$ for some integer $K$. Let $\tau = \pi_{\widehat{V}} \in \llbracket  \mathcal{G}_E  \vert_Y \rrbracket$ be the associated transposition. If $\alpha  \in \llbracket  \mathcal{G}_E  \vert_Y \rrbracket$ satisfies $\supp(\alpha) = s(V)$, then $\supp(\tau \alpha \tau) = r(V)$ and $S_{\tau \alpha \tau}(k) = \tau(S_\alpha(k))$ for each $k \in \ZZ$.
\end{lemma}
\begin{proof}
We first take care of the support of $\tau \alpha \tau$. If $x \notin r(V)$, then $\tau(x) \notin s(V) = \supp(\alpha)$. From this we see that $\tau \alpha \tau$ fixes $x$ because
\[ \tau \alpha \tau(x) = \tau \alpha (\tau(x)) = \tau \tau (x) = x.    \]
This shows that $\supp(\tau \alpha \tau) \subseteq r(V)$. By definition, the set $\{x \in \partial E \mid \alpha(x) \neq x \}$ is dense in $\supp(\alpha) = s(V)$. And then $Z \coloneqq \{ \tau(x) \mid x \in \partial E \ \& \ \alpha(x) \neq x \}$ is dense in $r(V)$. Let~$y \in Z$ and set $x = \tau(y)$, so that $y =\tau(x)$ and $\alpha(x) \neq x$. Then we have
\[\tau( \alpha(\tau(y))) = \tau( \alpha(\tau^2(x))) = \tau( \alpha(x)) \neq \tau(x) = y.   \]
Hence $Z \subseteq \supp(\tau \alpha \tau)  \subseteq r(V)$, and so by the density of $Z$ we get $\supp(\tau \alpha \tau) = r(V)$ as desired.

We now turn to the second statement. Let $x \in S_\alpha(k)$. Then~${(\alpha(x), k, x) \in U_\alpha}$. Consider first the case $x \in \supp(\alpha) = s(V)$. It is clear from the assumptions on~$V$ that we have $S_\tau(K) = s(V)$, $S_\tau(-K) = r(V)$ and that the rest is concentrated in~$S_\tau(0)$. Thus both $x$ and $\alpha(x)$ lie in $S_\tau(K)$. This means that $(\tau(x), K, x) \in U_\tau$ and that ${(\tau \alpha(x), K, \alpha(x)) \in U_\tau}$. Since $\tau = \tau^{-1}$ we also have $(\tau(x), K, x)^{-1} = (x,-K,\tau(x)) \in U_\tau$. Multiplying these together we obtain 
\[(\tau \alpha(x), K, \alpha(x)) \cdot (\alpha(x), k, x)  \cdot (x,-K,\tau(x)) = (\tau \alpha(x), k, \tau(x)) \in U_{\tau \alpha \tau},    \]
which shows precisely that $\tau(x) \in S_{\tau \alpha \tau}(k)$.

Lastly consider the case when $x \notin \supp(\alpha)$. Then we must have $k = 0$, and since~${\alpha(x) = x}$, we have $(x,0,x) \in U_\alpha$. If $x$ is not in the support of $\tau$ either (i.e.~${x \notin r(V)}$), then $\tau(x) = x \in S_{\tau \alpha \tau}(0)$ as desired. The final possibility is that $x \in r(V) = S_\tau(-K)$, and then $(\tau(x), -K, x) \in U_\tau$ and $(x, K, \tau(x)) \in U_\tau$. Multiplying these gives
\[ (\tau(x), -K, x)  \cdot (x,0,x)  \cdot (x, K, \tau(x)) = (\tau(x), 0, \tau(x)) \in U_{\tau \alpha \tau},    \]
hence $\tau(x) \in S_{\tau \alpha \tau}(0)$. 

We have shown that $\tau(S_\alpha(k)) \subseteq S_{\tau \alpha \tau}(k)$ for all $k$, but since both the $S_\alpha(k)$'s and the~$S_{\tau \alpha \tau}(k)$'s are partitions, we must actually have equality. This finishes the proof.
\end{proof}

\subsection{Identifying $I(\alpha)$}
Let us now turn to describing the image of the index map. Recall the homomorphism $\varphi \colon H_0(\mathcal{H}_E) \to H_0(\mathcal{H}_E)$ from Definition~\ref{def:delta}, where we view~$H_0(\mathcal{H}_E)$ as a subgroup of $H_0(\G_E\times_{c_E} \ZZ)$. For $n \in \NN$ its iterates are given by 
\[\varphi^n\left(\left[1_{Z(\mu) \times \{0\}}\right]\right) = \left[1_{Z(\mu) \times \{n\}}\right] =  \left[1_{Z(\nu \mu) \times \{0\}}\right],   \]
where $\nu$ is any path of length $n$ in $E$ terminating in $s(\mu)$. For any path~$\mu$ in~$E$ of length at least $n$ the iterated inverses are also defined, and they are given by 
\[\varphi^{-n}\left(\left[1_{Z(\mu) \times \{0\}}\right]\right) = \left[1_{Z(\mu) \times \{-n\}}\right] = \left[1_{Z(\sigma_E^n(\mu)) \times \{0\}}\right].   \]
In the setting of Definition~\ref{def:gradPart} we can write $U_\alpha \cap c_E^{-1}(k) =  \bigsqcup_{j=1}^{J_k} Z(\mu_j, F_j, \nu_j)$, where for each $j$, ${\vert \mu_j \vert - \vert \nu_j \vert = k}$. When $k < 0$ this entails that $\vert \nu_j \vert \geq \vert k \vert$. Since we have that~${S_\alpha(k) = s\left(U_\alpha \cap c_E^{-1}(k) \right) = \bigsqcup_{j=1}^{J_k} Z(\nu_j \setminus F_j)}$, the negative powers $\varphi^i$ are then defined on the associated characteristic functions for~${ - \vert k \vert \leq i \leq -1}$ and we have
\begin{equation}\label{eq:g}
\varphi^i \left( \left[ 1_{S_\alpha(k) \times \{0\}}    \right]\right) =  \left[ 1_{S_\alpha(k) \times \{i\}}    \right].
\end{equation}
For $k \geq 0$ and $i \geq 0$ Equation~\eqref{eq:g} clearly holds as well. For $i = k$ we furthermore have
\begin{equation}\label{eq:alpha}
\varphi^k \left( \left[ 1_{S_\alpha(k) \times \{0\}}    \right]\right) = \left[ 1_{\alpha(S_\alpha(k)) \times \{0\}}    \right].
\end{equation}

\begin{definition}
For $k \in \ZZ$ we define the following expression
\[\varphi^{(k)} \coloneqq   \begin{cases}
 -(\id + \varphi + \cdots + \varphi^{k-1}) & k > 0, \\
0 & k = 0, \\
 \varphi^{-1} + \varphi^{-2} + \cdots + \varphi^{k} & k < 0.
\end{cases} \]
\end{definition}

The definition above is somewhat formal in the sense that for $k < 0$ it is only defined on certain elements. However, we will only apply the negative powers as in Equation~\eqref{eq:g} where they are indeed defined. Observe that formally we have
\begin{equation}\label{eq:formal}
(\id - \varphi) \circ \varphi^{(k)} = \varphi^k - \id.
\end{equation}

Let us now show how an element $\alpha \in \llbracket \mathcal{G}_E \rrbracket$ gives rise to an element of $\ker(\id - \varphi)$ as on page~61 of~\cite{MatTFG}. Assume for simplicity that $E^0$ is finite, so that $S_\alpha(0)$ is compact.  Since both the $S_\alpha(k)$'s and $\alpha(S_\alpha(k))$'s form partitions of $\partial E$ we obtain the following using~\eqref{eq:alpha}
\[ \left[1_{\partial E} \right] =   \sum_{k \in \ZZ} \left[ 1_{S_\alpha(k) \times \{0\}}    \right] = \sum_{k \in \ZZ} \left[ 1_{\alpha(S_\alpha(k)) \times \{0\}}    \right] =  \sum_{k \in \ZZ} \varphi^k \left( \left[ 1_{S_\alpha(k) \times \{0\}}    \right]\right).\]
Subtracting these using~\eqref{eq:formal} we get 
\[ \sum_{k \in \ZZ} (\varphi^k - \id)\left( \left[ 1_{S_\alpha(k) \times \{0\}}    \right] \right) = (\id - \varphi) \left( \sum_{k \in \ZZ} \varphi^{(k)} \left( \left[ 1_{S_\alpha(k) \times \{0\}}    \right] \right) \right) = 0,  \]
which shows that $\sum_{k \in \ZZ} \varphi^{(k)} \left( \left[1_{S_\alpha(k) \times \{0\}} \right] \right) \in \ker(\id - \varphi)$. Analogously to Lemma~6.8 in~\cite{MatTFG} we will see that this is precisely the element to which $I(\alpha)$ corresponds. 

\begin{lemma}\label{lem:indeximage}
Let $E$ be an essential graph and let $\alpha = \pi_U \in \llbracket \mathcal{G}_E \rrbracket$. Under the identification $H_1(\mathcal{G}_E)  \cong \ker(\id - \varphi)$, the element $I(\alpha) \in H_1(\mathcal{G}_E)$ corresponds to
\[\sum_{k \in \ZZ} \varphi^{(k)} \left( \left[1_{S_\alpha(k) \times \{0\}} \right] \right) \in \ker(\id - \varphi) \leq H_0(\mathcal{H}_E). \] 
\end{lemma}
\begin{proof}
The identification
\(H_1(\mathcal{G}_E)  \cong \ker(\id - H_0(\rho_\bullet))\) from Proposition~\ref{prop:H01} is implemented by the (injective) connecting homomorphism $\partial_1 \colon H_1(\mathcal{G}_E) \to H_0(\mathcal{G}_E \times_{c_E} \ZZ)$ from the exact sequence~\eqref{eq:4exact}. Since $\ker(\id - \varphi) = \ker(\id - H_0(\rho_\bullet)) = \im(\partial_1)$ as subsets of $H_0(\mathcal{G}_E \times_{c_E} \ZZ)$, it suffices to compute $\partial_1(I(\alpha)) \in H_0(\mathcal{G}_E \times_{c_E} \ZZ)$. We will do this by stepwise going through the definition of $\partial_1$ in terms of the Snake Lemma applied to the diagram in Figure~\ref{fig:snake}. To save space we have shortened $C_c(\mathcal{G}, \ZZ)$ to $C_c(\mathcal{G})$ and $\mathcal{G}_E \times_{c_E} \ZZ$ to $\mathcal{G}_E \times \ZZ$. The maps~$\widetilde{\delta_1}$ in Figure~\ref{fig:snake} are given by $\widetilde{\delta_1}(f + \im(\delta_2)) = \delta_1(f)$. The top and bottom rows are the kernels and cokernels of the $\widetilde{\delta_1}$'s, respectively.
\begin{figure}
\[ \begin{tikzpicture}
  \matrix[matrix of math nodes,column sep={100pt,between origins},row
    sep={60pt,between origins},nodes={asymmetrical rectangle}] (s)
  {
    &|[name=ka]| 0 &|[name=kb]| 0 &|[name=kc]| H_1(\mathcal{G}_E) \\
    &|[name=A]| \frac{C_c(\mathcal{G}_E \times \ZZ)}{\im(\delta_2)} &|[name=B]| \frac{C_c(\mathcal{G}_E \times \ZZ)}{\im(\delta_2)} &|[name=C]| \frac{C_c(\mathcal{G}_E)}{\im(\delta_2)} &[-5em] |[name=01]| 0 \\
    |[name=02]| 0 &  |[name=A']| C_c(\partial E \times \ZZ) &|[name=B']| C_c(\partial E \times \ZZ) &|[name=C']| C_c(\partial E) \\
    &|[name=ca]| H_0(\mathcal{G}_E \times \ZZ) &|[name=cb]| H_0(\mathcal{G}_E \times \ZZ) &|[name=cc]| H_0(\mathcal{G}_E)  \\
  };
  \draw[->] (ka) edge (A)
            (kb) edge (B)
            (kc) edge[red] (C)
            (A) edge node[auto] {\tiny \( \id - \rho_1 / \im(\delta_2)  \)}      (B)				
            (B) edge[red] node[auto, red]   {\tiny \( \pi_1 / \im(\delta_2)  \)} (C)   
            (C) edge (01)
            (A) edge node[auto] {\tiny \(\widetilde{\delta_1}\)} (A')
            (B) edge[red] node[auto] {\tiny \(\widetilde{\delta_1}\)} (B')
            (C) edge node[auto] {\tiny \(\widetilde{\delta_1}\)} (C')
            (02) edge (A')
            (A') edge[red] node[auto] {\tiny \(\id - \rho_0\)} (B')
            (B') edge node[auto] {\tiny \(  \pi_0   \)} (C')
            (A') edge[red] (ca)
            (B') edge (cb)
            (C') edge (cc)
  ;
  \draw[->] (ka) edge[dashed] (kb)
                 (kb) edge[dashed] (kc)
                 (ca) edge[dashed] (cb)
                 (cb) edge[dashed] (cc)
  ;
  \draw[->,dashed,rounded corners,red] (kc) -| node[auto,text=black,pos=.7]
    {\color{red}{\(\partial_1\)}} ($(01.east)+(.5,0)$) |- ($(B)!.35!(B')$) -|
    ($(02.west)+(-.5,0)$) |- (ca);
\end{tikzpicture} \]
      \caption{The connecting homomorphism $\partial_1$ from the exact sequence~\eqref{eq:4exact}.}
      \label{fig:snake} 
\end{figure} 

We first treat the case when $E^0$ is finite, for then $U$ and $S_\alpha(0)$ are both compact. We start with $\alpha = \pi_U \in \llbracket \mathcal{G}_E \rrbracket$ and look at $I(\alpha) = [1_{U}] \in H_1(\mathcal{G}_E)$. Now view $1_U + \im(\delta_2)$ as an element of $C_c(\mathcal{G}_E) / \im(\delta_2)$ (recall that $\delta_1(1_U) = 0$). A lift of this element by $\pi_1 / \im(\delta_2)$ is given by the element $h \coloneqq 1_{U \times \{0\}} \in C_c(\mathcal{G}_E \times_{c_E} \ZZ)$, since $\pi_1(h) = 1_U$. At this point we have $h + \im(\delta_2) \in C_c(\mathcal{G}_E \times_{c_E} \ZZ) / \im(\delta_2)$. Before applying $\widetilde{\delta_1}$, we partition the full bisection $U$ defining $\alpha$ in terms of its values under the cocycle $c_E$:
\[U = \bigsqcup_{k=-N}^N U_k, \text{ where } U_k = U \cap c_E^{-1}(k), \]
so that $s\left(U_k \right) = S_\alpha(k)$. Note that
\begin{equation}\label{eq:unitspace}
1_{\partial E \times \{0\}} = \sum_{k=-N}^N  1_{s(U_k) \times \{0\} } = \sum_{k=-N}^N  1_{r(U_k) \times \{0\} }.
\end{equation}
Using this we compute
\begin{align*} 
\widetilde{\delta_1}(h + \im(\delta_2)) &= \delta_1(h) = \delta_1(1_{U \times \{0\}} ) = \sum_{k=-N}^N \delta_1(1_{U_k \times \{0\}} ) \\ 
&=  \sum_{k=-N}^N \left( s_\ast(1_{U_k \times \{0\}})   -   r_\ast(1_{U_k \times \{0\}})   \right)
 =   \sum_{k=-N}^N \left( 1_{s(U_k \times \{0\} )}   -   1_{r(U_k \times \{0\})}   \right) \\ 
 &=   \sum_{k=-N}^N \left( 1_{s(U_k) \times \{k\} }   -   1_{r(U_k) \times \{0\}}   \right) 
= \sum_{k=-N}^N \left( 1_{s(U_k) \times \{k\} }   -   1_{s(U_k) \times \{0\}}   \right) \\
 &= \sum_{k=-N}^{-1} \left( 1_{S_\alpha(k) \times \{k\} }   -   1_{S_\alpha(k) \times \{0\}}   \right) + \sum_{k=1}^{N} \left( 1_{S_\alpha(k) \times \{k\} }   -   1_{S_\alpha(k) \times \{0\}}   \right).
\end{align*}
The next step is to find the unique lift of $\delta_1(h)$ by $\id - \rho_0$. Applying Lemma~\ref{lem:SES1} to each term in the sum above we see that this lift is 
\[g \coloneqq        \sum_{k=-N}^{-1} \sum_{i= k}^{-1} 1_{S_\alpha(k) \times \{i\}}    -       \sum_{k=1}^{N} \sum_{i=0}^{k-1} 1_{S_\alpha(k) \times \{i\}}            \in C_c(\partial E \times \ZZ, \ZZ).   \]
The final step is to map the element $g$ ``downwards'' into the cokernel of $\widetilde{\delta_1}$, which is precisely $H_0(\mathcal{G}_E \times_{c_E} \ZZ)$. Using  Equation~\eqref{eq:g} we find
\begin{align*}
\partial_1(I(\alpha)) &= \partial_1([1_U]) = g + \im(\widetilde{\delta_1}) = [g] \\
& = \sum_{k=-N}^{-1} \sum_{i= k}^{-1} \left[ 1_{S_\alpha(k) \times \{i\}}   \right] -       \sum_{k=1}^{N} \sum_{i=0}^{k-1} \left[ 1_{S_\alpha(k) \times \{i\}}    \right]     \\
& = \sum_{k=-N}^{-1} \sum_{i=  k}^{-1} \varphi^i \left(  \left[ 1_{S_\alpha(k) \times \{0\}}   \right] \right)    -       \sum_{k=1}^{N} \sum_{i=0}^{k-1}  \varphi^i \left(  \left[ 1_{S_\alpha(k) \times \{0\}}   \right] \right)    \\
& = \sum_{k \in \ZZ} \varphi^{(k)} \left( \left[1_{S_\alpha(k) \times \{0\}} \right] \right)
\end{align*}

In the case that $E^0$ is infinite, the proof above remains valid if we simply replace~$U$ with~$U^\perp$ from Subsection~4.1, (as this makes all indicator functions above remain compactly supported) and replace~$\partial E$ with $\supp(\pi_U)$ in Equation~\eqref{eq:unitspace}.
\end{proof}

We emphasize that the sum in the lemma above really is a finite sum. Since we are aiming to establish Property~TR for restrictions of graph groupoids, we need to verify that the description of the index map as above also works in this case. 

\begin{corollary}\label{cor:indeximage}
Let $E$ be an essential graph and let~$Y \subseteq \partial E$ be clopen and $\mathcal{G}_E$-full. Then the element~$I(\alpha) \in H_1\left(\mathcal{G}_E  \vert_Y\right)$ for $\alpha \in \llbracket \mathcal{G}_E  \vert_Y \rrbracket$ corresponds to
 \[\sum_{k \in \ZZ} \varphi^{(k)} \left( \left[1_{S_\alpha(k) \times \{0\}} \right] \right) \in \ker(\id - \varphi) \leq H_0(\mathcal{H}_E) \]
 under the identification $\smash H_1\left(\mathcal{G}_E  \vert_Y\right)  \cong H_1(\mathcal{G}_E) \cong \ker(\id - \varphi)$, and the $S_\alpha(k)$'s form a finite clopen partition of $Y$. 
\end{corollary}
\begin{proof}
The inclusion $ \mathcal{G}_E  \vert_Y \hookrightarrow \mathcal{G}_E$ induces an isomorphism in homology due to the fullness of $Y$. We also have a canonical inclusion $\llbracket \mathcal{G}_E  \vert_Y \rrbracket \hookrightarrow \llbracket \mathcal{G}_E \rrbracket$ given by $\pi_U \mapsto \pi_{\widetilde{U}}$, where ~${\widetilde{U} = U \sqcup \partial E \setminus Y}$ for~$U \subseteq \mathcal{G}_E  \vert_Y$ a full bisection. In words, $\pi_{\widetilde{U}}$ simply extends $\pi_U$ trivially to the identity on $\partial E \setminus Y$. Together with the respective index maps, we claim that we from this get a commutative diagram as follows:
\[ \begin{tikzcd}
\llbracket \mathcal{G}_E  \vert_Y \rrbracket  \arrow{r}{I} \arrow[hook]{d}   &   H_1\left(\mathcal{G}_E  \vert_Y\right) \arrow{d}{\cong} \\ 
\llbracket \mathcal{G}_E \rrbracket  \arrow{r}{I}   &     H_1(\mathcal{G}_E)
\end{tikzcd} \]
To see that the diagram commutes, let $\alpha = \pi_U \in \llbracket \mathcal{G}_E  \vert_Y \rrbracket$ be given, where $U \subseteq \mathcal{G}_E  \vert_Y$ is a full bisection. The two paths in the diagram result in $\alpha \mapsto [1_{(\widetilde{U})^\perp}] \in H_1(\mathcal{G}_E)$ and~${\alpha \mapsto [1_{U^\perp}] \in H_1(\mathcal{G}_E)}$, respectively. But these elements are the same since the sets~$(\widetilde{U})^\perp$ and~$U^\perp$ are actually equal.

Let $\widetilde{\alpha} = \pi_{\widetilde{U}}$ denote the trivial extension of $\alpha$. Then $S_\alpha(k) = S_{\widetilde{\alpha}}(k)$ for all $k \neq 0$. Recall that $\varphi^{(0)} = 0$, so $k=0$ does not contribute. Appealing to Lemma~\ref{lem:indeximage} we obtain 
\[I(\alpha) \longleftrightarrow I(\widetilde{\alpha}) \longleftrightarrow \sum_{k \in \ZZ} \varphi^{(k)} \left( \left[1_{S_{\widetilde{\alpha}}(k) \times \{0\}} \right] \right) = \sum_{k \in \ZZ} \varphi^{(k)} \left( \left[1_{S_\alpha(k) \times \{0\}} \right] \right), \]
under the correspondence $\smash H_1\left(\mathcal{G}_E  \vert_Y\right)  \cong H_1(\mathcal{G}_E) \cong \ker(\id - \varphi) \leq H_0(\mathcal{H}_E)$.
\end{proof}

\begin{remark}\label{rem:deltadifference}
For finite graphs one might expect all formulas in the present paper to recover those in~\cite[Section~6]{MatTFG}  after substituting $\varphi = \delta^{-1}$, since one has that~${\ker(\id - \delta) = \ker(\id - \delta^{-1})}$ as sets. However, a small difference already appears in Corollary~\ref{cor:indeximage} when compared to~\cite[Lemma~6.8]{MatTFG}, which will propagate in the sequel. After substituting $\varphi = \delta^{-1}$, the $k$'th term (for $k \neq 0$) in Corollary~\ref{cor:indeximage} becomes 
\[\varphi^{(k)} =   \begin{cases}
 -(\id + \delta^{-1} + \cdots + \delta^{1-k}) & k > 0, \\
 \delta + \delta^{2} + \cdots + \delta^{ \vert k \vert } & k < 0,
\end{cases} \]
whereas the $k$'th term in~\cite[Lemma~6.8]{MatTFG} is
\[\delta^{(-k)} =   \begin{cases}
 -(\delta^{-1} + \delta^{-2}  \cdots + \delta^{-k}) & k > 0, \\
 \id + \delta + \cdots + \delta^{ \vert k \vert -1 } & k < 0.
\end{cases} \]
The reason these are different is because identifying $H_1(\mathcal{G}_E)$ with $\ker(\id - \delta)$ instead of~$\ker(\id - \delta^{-1})$ give different lifts of the element $\delta_1(h)$ in the proof of Lemma~\ref{lem:indeximage}. 
\end{remark}

\section{Establishing Property~TR}\label{sec:TR}

We are by now almost ready to prove that restrictions of graph groupoids have Property~TR. Given what we have established so far, our proof will in broad strokes follow the proof  of~\cite[Lemma~6.10]{MatTFG} using the endomorphism $\varphi$ instead of the automorphism~$\delta$ mentioned in Remark~\ref{rem:MatuiDelta}. However, there is another major difference, which we discuss below.  

What is actually proved in~\cite[Lemma~6.10]{MatTFG}  is that if the adjacency matrix~$A_E$ of a finite graph $E$ is \emph{primitive}\footnote{Meaning that for some $n \in \NN$ all entries in $(A_E)^n$ are strictly positive.}, then any restriction of $\mathcal{G}_E$ has Property~TR\footnote{At the beginning of the proof of~\cite[Theorem~6.11]{MatTFG} it is noted that the graph groupoid of a strongly connected finite graph is always Kakutani equivalent to graph groupoid whose adjacency matrix is primitive, from which it follows that restrictions of the former also have Property~TR.}. One reason why primitivity of the adjacency matrix is so useful is that this matrix then has a (strictly dominant) Perron eigenvalue $\lambda > 1$. Another is  that the AF-groupoid~$\mathcal{H}_E$ becomes minimal. This is if and only if, in fact, and also equivalent to the shift of finite type determined by $A_E$ being topologically mixing. In this case the infinite path space~$E^\infty$ admits exactly one $\mathcal{H}_E$-invariant probability measure. This measure, lets denote it by $\omega$, satisfies $\omega(s(U)) = \lambda \ \omega(rU))$ for any compact bisection $U \subseteq \mathcal{G}_E$ with~${U \subseteq c_E^{-1}(1)}$. This then allows one to compare clopen subsets and the image of the class of their characteristic functions under the automorphism $\delta$ and from this obtain bisections connecting them using~\cite[Lemma~6.7]{MatHom}. The approach in~\cite{MatTFG} was subsequently generalized to an abstract setting in~\cite[Proposition~4.5~(2)]{MatProd}.

In the setting of the present paper, however, where we allow infinite emitters in the graphs, we are no longer dealing with a shift of finite type (or any shift space for that matter), nor do we have a Perron eigenvalue. Neither is the AF-groupoid $\mathcal{H}_E$ ever minimal (see Remark~\ref{rem:fullness}). So the aforementioned~\cite[Proposition~4.5~(2)]{MatProd} does not apply. We replace the notion of primitivity (or mixing) by the technical Lemma~\ref{lem:longlemma} below. It prescribes necessary conditions on a graph~$E$ to guarantee the existence of certain disjoint paths in $E$ from which we can explicitly define sets with similar properties as the sets $C_{n,i}$ and $D_{n,j}$ which are constructed using the invariant measure $\omega$ in~\cite[Lemma~6.10]{MatTFG}. A key point is that these necessary conditions can always be arranged, without changing the isomorphism class of the groupoid, as demonstrated in Lemma~\ref{lem:moveT}.

\subsection{Technical lemmas}
The following ``combinatorial bookkeeping'' lemma will allow us to explicitly describe the terms in the sum in Corollary~\ref{cor:indeximage} and relate them to each other. As mentioned above, it will play a similar role as primitivity (or mixing) does in~\cite[Lemma~6.10]{MatTFG}.

\begin{lemma}\label{lem:longlemma}
Let $E$ be a strongly connected graph. Assume there is an infinite emitter in $E$ which supports infinitely many loops and from which there is at least one edge to any other vertex in $E$.  Let $\emptyset \neq Y \subseteq \partial E$ be clopen. Suppose we are given a clopen proper subset $\emptyset \neq A \subsetneq Y$, finite subsets $P \subset \NN$ and $Q \subset - \NN$, natural numbers $m_k \in \NN$ and vertices~$v_{k,i} \in E^0$ indexed over $k \in Q \cup \{0\} \cup P$ and $1 \leq i \leq m_k$. Then there exists a natural number $N \geq \max_{q \in Q} \vert q \vert$ and 
\begin{enumerate}
\item mutually disjoint paths $\gamma_{k,i}^{(0)} \in E^N v_{k,i}$ such that $Z\left(\gamma_{k,i}^{(0)}\right) \subseteq Y \setminus A$ for all $k$ in ${Q \cup \{0\} \cup P}$ and $1 \leq i \leq m_k$,
\item mutually disjoint paths $\gamma_{p,i}^{(j)} \in E^{N+j} v_{p,i}$ such that $Z\left(\gamma_{p,i}^{(j)} \right) \subseteq A$ for all $p \in P$, $1 \leq i \leq m_p$ and $j = 1, 2, \ldots, p$,
\item  mutually disjoint paths $\gamma_{q,i}^{(l)} \in E^{N-l} v_{q,i}$ such that $Z\left(\gamma_{q,i}^{(l)} \right) \subseteq A$ for all $q \in Q$, $1 \leq i \leq m_q$ and $l = 1, 2, \ldots, \vert q \vert$.
\end{enumerate}
\end{lemma}
\begin{proof}
Pick an infinite emitter~$w \in E^0_{\text{sing}}$ which satisfy the assumptions in the lemma. We enumerate the infinitely many loops based at $w$ as $e_{k,i}$ (these are all distinct) where~$k$ and~$i$ both range over $\NN$. Choose paths $\mu, \mu' \in E^*$ such that $Z(\mu) \subseteq Y \setminus A$ and~$Z(\mu') \subseteq A$. By extending these paths we may assume that they both end in $w$, and by concatenating sufficiently many loops at $w$ to the shortest one of these, we may furthermore assume that $\vert \mu \vert = \vert \mu' \vert$. For each $k \in Q \cup \{0\} \cup P$ and $1 \leq i \leq m_k$ we pick an edge $f_{k,i}$ which goes from $w$ to $v_{k,i}$.

The paths we will define will all start with either $\mu$ or $\mu'$, which will ensure that their cylinder sets are contained in either $A$ or $Y \setminus A$ as needed. Then they will have a certain number of the loops at $w$ and it is these that will ensure the paths are mutually disjoint. And they will all end with an edge $f_{k,i}$ taking care of the range of the paths. We set~$K \coloneqq \max_{q \in Q} \vert q \vert$ and $M \coloneqq \vert \mu \vert = \vert \mu' \vert$, and then define $N \coloneqq M + K + 2$. Here $M$ is present because all the paths start with $\mu$ or $\mu'$, $K$ is a ``buffer'' we can subtract from for the $\gamma_{q,i}^{(l)}$'s (as these should have length $N - l$) and the term $2$ comes from having at least one loop $e_{k,i}$ and then ending with $f_{k,i}$. We now define the desired paths as follows: 

\begin{alignat*}{3}
&(1) \ \gamma_{k,i}^{(0)} \coloneqq \mu \ e_{k,i}^{K+1} \ f_{k,i} && \text{ for } k \in Q \cup \{0\} \cup P \text{ and } 1 \leq i \leq m_k, \\
&(2) \ \gamma_{p,i}^{(j)} \coloneqq \mu' \ e_{p,i}^{K+1 + j} \ f_{p,i} && \text{ for } p \in P, \ 1 \leq i \leq m_p \text{ and }   j = 1, 2, \ldots, p, \\
&(3) \ \gamma_{q,i}^{(l)} \coloneqq \mu' \ e_{q,i}^{K+1 - l} \ f_{q,i} \qquad && \text{ for } q \in Q, \ 1 \leq i \leq m_q \text{ and }   l = 1, 2, \ldots, \vert q \vert.
\end{alignat*}
It is clear that these satisfy the conditions in the lemma.
\end{proof} 

The next lemma shows that for a graph $E$ with finitely many vertices, the conditions in Lemma~\ref{lem:longlemma} can always be arranged, by changing the graph, but without changing the (isomorphism class of the) groupoid. This is actually the only place where we need to assume that the graph has finitely many vertices (see also Remark~\ref{rem:infiniteVertex}).

In order to prove it, we will appeal to one of Sørensen's geometric moves on graphs from~\cite{Sorensen}. On page 1207 therein, a move on graphs called \emph{move~(T)} is described. In order to apply this move one needs a graph $E$ with an infinite emitter $w \in E^0_\text{sing}$. If there is a path $e_1 e_2 \cdots e_n$ in $E$ from $w$ to another vertex $v$ such that $w$ emits infinitely many edges to $r(e_1)$, then move~(T) is the operation of adding a countably infinite number of new edges from $w$ to $v$.

It is proved in~\cite[Theorem~5.4]{Sorensen} that move~(T) can be expressed using the first four ``standard moves'' in Section~3 of~\cite{Sorensen}. This means that move~(T) produces \emph{move equivalent} graphs, which in turn implies that the associated graph groupoids are Kakutani equivalent~\cite{CRS}. But by virtue of~\cite[Lemma~6.5]{BCW} we can deduce something even stronger, namely that move~(T) actually produce \emph{orbit equivalent} graphs. And in our setting this in fact implies isomorphism of the graph groupoids.

\begin{lemma}\label{lem:moveT}
Let $E$ be a strongly connected graph with finitely many vertices and suppose that $E$ has an infinite emitter $w \in E^0_\text{sing}$. Let $F$ denote the graph which is obtained from~$E$ by, for each $v \in E^0$, adding a countably infinite number of new edges from $w$ to~$v$. Then~$\mathcal{G}_E \cong \mathcal{G}_F$ as étale groupoids.  
\end{lemma}
\begin{proof}
The strong connectedness of $E$ guarantees, that for each vertex $v \in E^0$, there exists a path from $w$ to $v$ that starts with an edge to a vertex to which $w$ emits infinitely already. Thus we see that the graph $F$ is obtained from $E$ by applying move~(T) finitely many times. As mentioned in the paragraph above, this implies that the graphs~$E$ and~$F$ are orbit equivalent. The assumptions on $E$ also ensure that $E$ satisfies Condition~(L), and therefore so does $F$. It now follows from the main result of~\cite{BCW} that~${\mathcal{G}_E \cong \mathcal{G}_F}$.
\end{proof}

The final lemma describes in some sense a ``graded cancellation'' for the map $\varphi$ on~$H_0(\mathcal{H}_E)$. It is a straightforward extension of~\cite[Lemma~6.9]{MatTFG}, after having established cancellation for general AF-groupoids in Section~\ref{sec:cancellation}, but we have nevertheless included the short argument for completeness.

\begin{lemma}\label{lem:gradedcancellation}
Let $E$ be an essential graph and let $A,B \subseteq \partial E$ be compact open subsets. If $\varphi^n \left( \left[1_{A } \right] \right) = \left[1_{B } \right]$ in $H_0(\mathcal{H}_E)$ for some $n \in \NN$, then there exists a bisection $U \subseteq \mathcal{G}_E$ satisfying $U \subseteq c_E^{-1}(n)$, $s(U) = A$ and $r(U) = B$.
\end{lemma}
\begin{proof}
We first write $A$ as a disjoint union of punctured sylinder sets $A = \sqcup_{j=1}^J Z(\mu_j \setminus F_j)$. Now pick paths $\gamma_j \in E^n$ with $r(\gamma_j) = s(\mu_j)$ and set $C \coloneqq \sqcup_{j=1}^J Z(\gamma_j \mu_j \setminus F_j)$. Then we have 
\[\left[1_{B } \right] = \varphi^n \left( \left[1_{A} \right] \right) = \left[1_{C } \right] \text{ in } H_0(\mathcal{H}_E) \]
by definition of $\varphi$. Invoking cancellation in the AF-groupoid $\mathcal{H}_E$ (Theorem~\ref{thm:cancellation}) produces a bisection $W \subseteq \mathcal{H}_E \subseteq \mathcal{G}_E$ with $s(W) = C$ and $r(W) = B$. Next define the bisection~${V \coloneqq  \sqcup_{j=1}^J Z(\gamma_j \mu_j, F_j, \mu_j)}$, which satisfies $s(V) = A$ and $r(V) = C$. Finally, setting~${U \coloneqq WV}$ gives us the desired bisection since $s(U) = s(V) = A$, $r(U) = r(W) = B$ and $U \subseteq c_E^{-1}(n)$, because $W \subseteq c_E^{-1}(0)$ and $V \subseteq c_E^{-1}(n)$.
\end{proof}

\subsection{The main result}
We are now ready to give the proof of our main result.
\begin{theorem}\label{thm:TR}
Let $E$ be a strongly connected graph with finitely many vertices and at least one infinite emitter. Let further $\emptyset \neq Y \subseteq \mathcal{G}_E^{(0)} = \partial E$ be clopen. Then the restricted graph groupoid $\mathcal{G}_E  \vert_Y$ has Property~TR.
\end{theorem}
\begin{proof} 
Let $\alpha = \pi_U \in \llbracket \mathcal{G}_E  \vert_Y \rrbracket \setminus \{ \id \}$ be given and suppose that $I(\alpha) = 0$ in $H_1(\mathcal{G}_E  \vert_Y)$. We are going to show that $\alpha$ is a product of transpositions. In the previous section we saw that $I(\alpha)$ corresponds to an element in $\ker(\id - \varphi) \leq H_0(\mathcal{H}_E)$ which is described in terms of the finite clopen partition $\{ S_\alpha(k) \}_{k \in \ZZ}$ of $Y$. Define 
\[P \coloneqq \{ k > 0 \mid S_\alpha(k) \neq \emptyset \} \qquad \text{and} \qquad  Q \coloneqq \{ k < 0 \mid S_\alpha(k) \neq \emptyset \}.  \]
These are finite subsets of $\NN$. Set $A \coloneqq \supp(\alpha)$. By Lemma~\ref{lem:support} we may assume that~${A \neq Y}$. And $A$ is non-empty since $\alpha \neq \id$. We can write 
\[A = \supp(\alpha) =  (S_\alpha(0) \cap A) \bigsqcup \bigsqcup_{k \in Q \cup P} S_\alpha(k).  \]
Now decompose these in terms of punctured cylinder sets as 
\[ S_\alpha(0) \cap A = \bigsqcup_{i=1}^{m_0} Z(\mu_{0,i} \setminus F_{0,i})     \qquad \text{and} \qquad   S_\alpha(k) =  \bigsqcup_{i=1}^{m_k} Z(\mu_{k,i} \setminus F_{k,i}),    \]
where $\mu_{k,i} \in E^*$ and $F_{k,i} \subseteq_{\text{finite}} r(\mu_{k,i})$. It is possible for one of $P$, $Q$ or $S_\alpha(0) \cap A$ to be empty (but not all of them). For now we assume that all three are non-empty, and we shall comment on what happens otherwise near the end of the proof.   

At this point we want to invoke Lemma~\ref{lem:longlemma}. By Lemma~\ref{lem:moveT} we may assume that $E$ satisfies the hypotheses of Lemma~\ref{lem:longlemma}. Setting $v_{k,i} = s(\mu_{k,i})$ in Lemma~\ref{lem:longlemma} gives us a natural number $N$ (larger in absolute value than all numbers in $Q$) and
\begin{enumerate}
\item mutually disjoint paths $\gamma_{k,i}^{(0)} \in E^N s(\mu_{k,i})$ such that $Z\left(\gamma_{k,i}^{(0)}\right) \subseteq Y \setminus A$ for all $k \in Q \cup \{0\} \cup P$ and $1 \leq i \leq m_k$,
\item mutually disjoint paths $\gamma_{p,i}^{(j)} \in E^{N+j} s(\mu_{p,i})$ such that $Z\left(\gamma_{p,i}^{(j)} \right) \subseteq A$ for all $p \in P$, $1 \leq i \leq m_p$ and $j = 1, 2, \ldots, p$,
\item  mutually disjoint paths $\gamma_{q,i}^{(l)} \in E^{N-l} s(\mu_{q,i})$ such that $Z\left(\gamma_{q,i}^{(l)} \right) \subseteq A$ for all $q \in Q$, $1 \leq i \leq m_q$ and $l = 1, 2, \ldots, \vert q \vert$.
\end{enumerate}
From these we define the compact open set
\[ B \coloneqq    \bigsqcup_{k \in Q \cup \{0\} \cup P} \       \bigsqcup_{i=1}^{m_k} Z(\gamma_{k,i}^{(0)} \mu_{k,i} \setminus F_{k,i} ) \subseteq Y \setminus A. \]
Next we define the bisection 
\[V \coloneqq \bigsqcup_{k \in Q \cup \{0\} \cup P} \        \bigsqcup_{i=1}^{m_k} Z(\gamma_{k,i}^{(0)}  \mu_{k,i}, F_{k,i},  \mu_{k,i}).   \]
As $s(V) = A$ is disjoint from $r(V) = B$ we get a transposition
${\tau_V \coloneqq \pi_{\widehat{V}} \in \llbracket \mathcal{G}_E  \vert_Y \rrbracket}$. This transposition satisfies $\tau_V(A) = B$, $\tau_V(B) = A$, ${\supp(\tau_V) = A \sqcup B}$ and 
\[  S_{\tau_V}(N) = A, \quad   S_{\tau_V}(-N) = B,  \quad   S_{\tau_V}(0) = Y \setminus A \sqcup B.          \]

We now define another element in $\llbracket \mathcal{G}_E  \vert_Y \rrbracket$, namely $\beta \coloneqq \tau_V  \alpha  \tau_V$. If we can prove that~$\beta$ is a product of transpositions, then the theorem follows. To do just that, we are going construct another element $\tau \in \llbracket \mathcal{G}_E  \vert_Y \rrbracket$, which is itself a product of transpositions, but which also satisfies $S_\tau(k) = S_\beta(k)$ for all $k$.  The construction of $\tau$ is a bit involved, so before we get to that, let us explain why having $\tau$ suffices. Given an element $\tau$ as above, we deduce from Lemma~\ref{lem:betaalphainverse} that $\beta \tau^{-1} \in  \llbracket \mathcal{H}_E  \vert_Y \rrbracket$. Making use of the fact that~${I(\alpha) = 0}$ we find that $I(\beta \tau^{-1}) =0$ as well. Indeed, \[I(\beta \tau^{-1}) = I(\tau_V \alpha \tau_V \tau^{-1}) = I(\tau_V) + I(\alpha) + I(\tau_V) - I(\tau),\]
which are all $0$ as transpositions are in the kernel of the index map. The groupoid~$\mathcal{H}_E  \vert_Y$ is an AF-groupoid, and since all AF-groupoids have Property~TR~\cite[Theorem~3.3.(4)]{MatProd} we deduce that $\beta \tau^{-1}$ is a product of transpositions (in $\llbracket \mathcal{H}_E  \vert_Y\rrbracket) \subseteq \llbracket \mathcal{G}_E  \vert_Y \rrbracket $). But then $\beta$ is a product of transpositions as well.

All that remains now is the construction of $\tau$ as above. The element~$\tau$ will be of the form $\tau =   \tau_{-} \circ \tau_{+}$, where $\tau_{+}$ will be constructed from the $S_\beta(p)$'s for $p \in P$ and similarly~$\tau_{-}$ comes from the $S_\beta(q)$'s for $q \in Q$. We begin by noting that $\supp(\beta) = B$ and that
\begin{equation}\label{eq:suppBeta}
S_\beta(k) = \tau_V(S_\alpha(k)) = \begin{cases}
\bigsqcup_{i=1}^{m_k} Z(\gamma_{k,i}^{(0)} \mu_{k,i} \setminus F_{k,i}) & \text{ for } k \neq 0 \\
\bigsqcup_{i=1}^{m_0} Z(\gamma_{0,i}^{(0)} \mu_{0,i} \setminus F_{0,i}) \bigsqcup Y \setminus B & \text{ for } k = 0
\end{cases}
\end{equation}
by Lemma~\ref{lem:Sbeta}. Let us define the compact open sets 
\[D_{p,j} \coloneqq \bigsqcup_{i=1}^{m_p} Z(\gamma_{p,i}^{(j)} \mu_{p,i} \setminus F_{p,i})  \]
for $p\in P$ and $1 \leq j \leq p$ and set
\[D \coloneqq \bigsqcup_{p \in P} \left(  \bigsqcup_{j=1}^{p-1}  D_{p,j}  \bigsqcup  S_\beta(p) \right).   \] 
Observe that\footnote{Henceforth we suppress the ``$\times \{0\}$'' from $S_\beta(0) \times \{0\}$ to increase readability.} 
\begin{equation}\label{eq:phiD}
\varphi^{j} \left( \left[1_{S_\beta(p)} \right] \right) =    \left[1_{D_{p,j}} \right]  \in H_0(\mathcal{H}_E) 
\end{equation}
for $p,j$ as above. Furthermore, for $p \in P$ define the bisections 
\[W_{p,j}  \coloneqq \bigsqcup_{i=1}^{m_p} Z(\gamma_{p,i}^{(j)} \mu_{p,i}, \ F_{p,i}, \ \gamma_{p,i}^{(j-1)} \mu_{p,i}) \subseteq \mathcal{G}_E \quad \text{ for }  1 \leq j \leq p.\]
Using Equation~\eqref{eq:suppBeta} and the definition of the $D_{p,j}$'s we observe that 
\begin{align*}
W_{p,j} \subseteq  c_E^{-1}(-1),& \quad r(W_{p,j}) = D_{p,j} \text{ for } j \geq 1, \\
s(W_{p,1}) = S_\beta(p),& \quad s(W_{p,j}) = D_{p,j-1}  \text{ for } j \geq 2.
\end{align*}
As the sources and ranges of these bisections are disjoint (mutually disjoint even) we obtain transpositions $\tau_{p,j} = \pi_{\widehat{W_{p,j}}}$ which swap them. Now we are ready to define the ``first half'' of $\tau$, namely $\tau_{+}$, as follows 
\[\tau_{+} \coloneqq \prod_{p \in P} \tau_{p,p} \circ \tau_{p, p-1} \circ \cdots \circ \tau_{p, 1}.   \]
As a homeomorphism, $\tau_{+}$ is the ``disjoint union of the cycles'' \[S_\beta(p) \mapsto D_{p,p} \mapsto D_{p,p-1} \mapsto \cdots \mapsto D_{p,1} \mapsto S_\beta(p)\] for $p \in P$. Observe that we have
\begin{align*}
\tau_{+}(S_\beta(p)) &= D_{p,p}, \quad S_{\tau_{+}}(p) =  S_\beta(p), \quad S_{\tau_{+}}(-1) = \bigsqcup_{p \in P}   \bigsqcup_{j=1}^{p}  D_{p,j}, \\
\supp(\tau_{+}) &= \bigsqcup_{p \in P} S_{\tau_{+}}(p) \bigsqcup S_{\tau_{+}}(-1) = D \bigsqcup_{p \in P}  D_{p,p}.
\end{align*} 

Our next objective is to construct the other half of $\tau$, namely $\tau_{-}$. Combining Corollary~\ref{cor:indeximage} ($Y$ is full because $\mathcal{G}_E$ is minimal) with Equation~\eqref{eq:phiD} we obtain
\begin{align}\label{eq:homology2}
 0 &= I(\alpha) = I(\beta) = \sum_{k \in \ZZ} \varphi^{(k)} \left( \left[1_{S_\beta(k)} \right] \right) = \sum_{k \in \ZZ \setminus \{0\}} \varphi^{(k)} \left( \left[1_{S_\beta(k)} \right] \right) \nonumber \\
& \implies  \sum_{q \in Q} \varphi^{(q )} \left( \left[1_{S_\beta(q)} \right] \right) = - \sum_{p \in P} \varphi^{(p)} \left( \left[1_{S_\beta(p)} \right] \right) \nonumber \\
& \implies  \sum_{q \in Q} \left( \varphi^{-1} \left( \left[1_{S_\beta(q)} \right] \right) + \varphi^{-2} \left( \left[1_{S_\beta(q)} \right] \right) + \cdots + \varphi^{ q } \left( \left[1_{S_\beta(q)} \right] \right)    \right) \nonumber \\
&=  \sum_{p \in P} \left(  \left[1_{S_\beta(p)} \right] + \varphi \left( \left[1_{S_\beta(p)} \right] \right) + \cdots + \varphi^{p-1} \left( \left[1_{S_\beta(p)} \right] \right)    \right) \nonumber  \\
& = \sum_{p \in P} \left(  \left[1_{S_\beta(p)} \right] + \left[1_{D_{p,1}} \right] + \cdots + \left[1_{D_{p,p-1}} \right]   \right) = \left[  1_D  \right].
\end{align}
Similarly to the $D_{p,j}$'s, we define the compact open sets
\[X_{q,l} \coloneqq \bigsqcup_{i=1}^{m_q} Z(\gamma_{q,i}^{(l)} \mu_{q,i} \setminus F_{q,i})  \]
for $q \in Q$ and $1 \leq j \leq \vert q \vert$, and set
\[X \coloneqq \bigsqcup_{q \in Q}  \bigsqcup_{l=1}^{\vert q \vert}  X_{q,l}.   \] 
These sets then satisfy 
\begin{equation}\label{eq:phiX}
\varphi^{-l} \left( \left[1_{S_\beta(q)} \right] \right) =    \left[1_{X_{q,l}} \right] \in H_0(\mathcal{H}_E)
\end{equation}
for $q,l$ as above.
Equation~\eqref{eq:homology2} now says that $\left[  1_X  \right]= \left[  1_D  \right]$ in $H_0(\mathcal{H}_E)$. Invoking cancellation in $\mathcal{H}_E$ (Theorem~\ref{thm:cancellation}) we can find a bisection $R \subseteq \mathcal{H}_E$ with $s(R) = X$ and $r(R) = D$. Now define $R_{q,l} \coloneqq s^{-1}\left(X_{q,l} \right)$ and $C_{q,l} \coloneqq r\left(R_{q,l} \right)$. Then the $R_{q,l}$'s are mutually disjoint bisections which witness that $\left[  1_{C_{q,l}}  \right]= \left[  1_{X_{q,l}}  \right]$. We also define 
\[C \coloneqq \bigsqcup_{q \in Q}  \bigsqcup_{l=1}^{\vert q \vert}  C_{q,l}.   \] 
Observe that we actually have $C = D$, as they both equal $r(R)$. Equation~\eqref{eq:phiX} implies that 
\[ \varphi \left( \left[1_{C_{q,1}} \right] \right) = \left[1_{S_\beta(q)} \right]   \quad \text{and} \quad \varphi \left( \left[1_{C_{q,l}} \right] \right) = \left[1_{C_{q,l-1}} \right] \ \text{ for } l \geq 2 \]
in $H_0(\mathcal{H}_E)$. Hence Lemma~\ref{lem:gradedcancellation} yields bisections $T_{q,l} \subseteq  \mathcal{G}_E$ satisfying
\begin{align*}
T_{q,l} \subseteq c_E^{-1}(1),& \quad s(T_{q,l}) = C_{q,l} \text{ for } l \geq 1, \\
r(T_{q,1}) = S_\beta(q),& \quad r(T_{q,l}) = C_{q,l-1} \text{ for } l \geq 2.
\end{align*}
Let $\tau_{q,l} \coloneqq \pi_{\widehat{T_{q,l}}}$ denote the associated transpositions. From these we in turn define $\tau_{-}$ in a similar fashion as $\tau_{+}$ by setting
\[\tau_{-} \coloneqq \prod_{q \in Q} \tau_{q, \vert q \vert} \circ \tau_{q, \vert q \vert -1} \circ \cdots \circ \tau_{q, 1}.   \]
Just like $\tau_{+}$, the homeomorphism $\tau_{-}$ is a ``disjoint union of cycles'' \[S_\beta(q) \mapsto C_{q,\vert q \vert} \mapsto C_{q,\vert q \vert -1} \mapsto \cdots \mapsto C_{q,1} \mapsto S_\beta(q)\] for $q \in Q$. And we have
\begin{align*}
\tau_{-}(S_\beta(q)) &= C_{q,\vert q \vert}, \quad S_{\tau_{-}}(q) =  S_\beta(q), \quad S_{\tau_{-}}(1) = \bigsqcup_{q \in Q}  \bigsqcup_{l=1}^{\vert q \vert}  C_{q,l} = C, \\
\supp(\tau_{-}) &= \bigsqcup_{q \in Q} S_{\tau_{-}}(q) \bigsqcup S_{\tau_{-}}(1) = C \bigsqcup_{q \in Q}  S_\beta(q). 
\end{align*}  

Finally, we define $\tau \coloneqq \tau_{-} \circ \tau_{+}$. In order to finish the proof, we need to show that~${S_{\tau}(k) = S_\beta(k)}$ for all~$k \in \ZZ$. We start by noting that 
\[ \supp(\tau) \subseteq \supp(\tau_{+}) \bigcup \supp(\tau_{-}) = \left( \bigsqcup_{q \in Q} S_\beta(q) \right) \bigsqcup \left( \bigsqcup_{p \in P} S_\beta(p) \right) \bigsqcup \left(  \bigsqcup_{p \in P} \bigsqcup_{j=1}^p D_{p,j} \right).    \]
We are going to analyze this support piece by piece. We begin by fixing some $q \in Q$ and consider $S_\beta(q)$. Since $S_\beta(q) \subseteq Y \setminus \supp(\tau_{+})$ we have 
\[ S_\beta(q) \xmapsto[\text{lag } 0]{\tau_{+}} S_\beta(q) \xmapsto[\text{lag } q]{\tau_{-}} C_{q, \vert q \vert}. \]
This means that $S_\beta(q) \subseteq S_\tau(q)$. We similarly have $S_\beta(p) \subseteq S_\tau(p)$ for each $p \in P$ since~${D_{p,p} \subseteq Y \setminus \supp(\tau_{-})}$.
For the last part, we consider the sets $D_{p,j}$ for~$p \in P$ and~${1 \leq j \leq p}$. For $j = 1$ we find that 
\[ D_{p,1} \xmapsto[\text{lag } -1]{\tau_{+}} S_\beta(p) \xmapsto[\text{lag } 1]{\tau_{-}} \tau_{-}(S_\beta(p)) \]
because $ S_\beta(p) \subseteq D = C$, which maps with lag $1$ by $\tau_{-}$. As the total lag is $1-1 = 0$, we get that $D_{p,1} \subseteq S_\tau(0)$. When $j \geq 2$ we similarly have
\[ D_{p,j} \xmapsto[\text{lag } -1]{\tau_{+}} D_{p,j-1} \xmapsto[\text{lag } 1]{\tau_{-}} \tau_{-}(D_{p,j-1}) \]
since $D_{p,j-1} \subseteq C$. Hence $D_{p,j} \subseteq S_\tau(0)$ as well. If we now set \[Z \coloneqq \left( \bigsqcup_{q \in Q} S_\beta(q) \right) \bigsqcup \left( \bigsqcup_{p \in P} S_\beta(p) \right) \bigsqcup \left(  \bigsqcup_{p \in P} \bigsqcup_{j=1}^p D_{p,j} \right) \subseteq Y \setminus \supp(\tau)  \] and decompose $Y$ as
\[Y = \left( \bigsqcup_{q \in Q} S_\beta(q) \right) \bigsqcup \left( \bigsqcup_{p \in P} S_\beta(p) \right) \bigsqcup \left(  \bigsqcup_{p \in P} \bigsqcup_{j=1}^p D_{p,j} \right) \bigsqcup \left( Y \setminus Z \right)  \]
then we have seen that 
\[ S_\beta(q)  \subseteq S_\tau(q), \quad S_\beta(p)  \subseteq S_\tau(p), \quad D_{p,j}  \subseteq S_\tau(0), \quad Y \setminus Z   \subseteq S_\tau(0). \]
Since both of these form partitions of $Y$ we must actually have equality here. This means that $S_\beta(k) = S_\tau(k)$ for all $k \neq 0$. And then $S_\beta(0) = S_\tau(0)$ as well.

Let us now comment on what happens if one of $P$, $Q$ or $S_\alpha(0) \cap A$ are empty. All three cannot be empty since $\supp(\alpha) \neq \emptyset$. And we claim that $P = \emptyset$ if and only if $Q = \emptyset$. Arguing by contradicition, if $P \neq \emptyset$ and $Q = \emptyset$, then Equation~\eqref{eq:homology2} says that $[1_D] = 0$ in $H_0(\mathcal{H}_E)$, so by Corollary~\ref{cor:zeroCancel} $D = \emptyset$. But this forces $P = \emptyset$. Having $P = \emptyset$ and $Q \neq \emptyset$ is ruled out similarly. In the case of $P = \emptyset = Q$ we have that $A = \supp(\alpha) \subseteq S_\alpha(0)$, which means that $\alpha \in \llbracket \mathcal{H}_E  \vert_Y \rrbracket$ (since $U_\alpha \subseteq  c_E^{-1}(0)$). And then we are done since this groupoid is AF and hence has Property~TR. The last possibility is that $S_\alpha(0) \cap A = \emptyset$ and $P, Q$ are both non-empty. In this case the proof above goes through by removing everything indexed by $k=0$. This finishes the proof at large.
\end{proof}

Having established Property~TR for strongly connected graphs with infinite emitters, we deduce the AH~conjecture for these from~\cite[Theorem~4.4]{MatProd}. But as we saw in 
Proposition~\ref{prop:AH} 
the assumptions in the AH~conjecture for graph groupoids are slightly weaker than strong connectedness. For completeness we want to show that all graph groupoids covered by the assumptions satisfy the conjecture. Using another one of Sørensen's moves on graphs, namely \emph{source removal}, we can actually reduce this to the strongly connected situation.  

\begin{corollary}\label{cor:AH}
Let $E$ be a graph satisfying the AH~criteria and let $Y \subseteq \mathcal{G}_E^{(0)} = \partial E$ be clopen. Then the AH~conjecture is true for $\mathcal{G}_E  \vert_Y$. 
\end{corollary}
\begin{proof}
As discussed in Subsection~4.2, the graph $E$ has a single nontrivial strongly connected component which contain all infinite emitters. The vertices which lie outside this component and the edges they emit form an acyclic subgraph with sources which connect to the nontrivial connected component. By repeatedly applying Sørensen's move~(S) from~\cite[Section~3]{Sorensen} we can remove all the vertices lying outside the strongly connected component of $E$. This results in a graph $F$ which is strongly connected and which is move equivalent to $E$. By the results in~\cite{CRS}~$\mathcal{G}_F$ is Kakutani equivalent to~$\mathcal{G}_E$. Hence there are full clopen subsets~${X \subseteq \mathcal{G}_E^{(0)}}$ and $Z \subseteq \mathcal{G}_F^{(0)}$ such that $\mathcal{G}_E  \vert_X \cong \mathcal{G}_F  \vert_Z$. Appealing to~\cite[Proposition 4.11]{MatTFG} we can find a compact bisection  $U \subseteq \mathcal{G}_E$ satisfying $s(U) = Y$ and $r(U) \subseteq X$. And then 
\[ \mathcal{G}_E  \vert_Y \cong \mathcal{G}_E \vert_{r(U)} = \left( \mathcal{G}_E  \vert_X \right) \vert_{r(U)} \cong \left( \mathcal{G}_F  \vert_Z \right) \vert_{Z'}  =  \mathcal{G}_F \vert_{Z'}   \]
for some clopen set $Z' \subseteq Z \subseteq \mathcal{G}_F^{(0)}$.

If $E$ has infinite emitters, then the result follows from applying Theorem~\ref{thm:TR} to~$\mathcal{G}_F \vert_{Z'}$. And if $E$ is finite it similarly follows from the results in~\cite[Subsection~6.4]{MatTFG}. 
\end{proof} 


\begin{remark}\label{rem:infiniteVertex}
The finiteness assumption on the set of vertices is actually only needed to guarantee that we can apply Lemma~\ref{lem:longlemma}, by first applying Lemma~\ref{lem:moveT}. Hence Theorem~\ref{thm:TR} also applies to strongly connected graphs with infinitely many vertices, provided that the graph satisfies the hypotheses of Lemma~\ref{lem:longlemma}. Namely that there exists an infinite emitter which supports infinitely many loops and from which there is at least one edge to any other vertex.
\end{remark}


\section{Examples and applications}\label{sec:examples}

\subsection{Groupoid models for Cuntz algebras}
Let $E_n$ denote the graph with one vertex and $n$ loops for $2 \leq n \leq \infty$.
The graph \mbox{$C^*$-algebras} of these graphs are the Cuntz algebras, that is~${C^*(E_n) \cong \mathcal{O}_n}$, whose $K$-theory is given by $\mathbb{Z}_n$ and $0$ respectively (where~$\mathbb{Z}_\infty$ means~$\ZZ$). 

Let us now consider our main motivating example, namely the graph
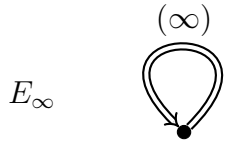
\[ \begin{tikzpicture}[vertex/.style={circle, draw = black, fill = black, inner sep=0pt,minimum size=5pt}, implies/.style={double,double equal sign distance,-implies}]
\node at (-2,0.5) {$E_\infty$};
\node[vertex] (a) at (0,0) {};

\path (a) 	edge[implies, thick, loop, min distance = 20mm, looseness = 10, out = 45, in = 135]  node[above] {$(\infty)$} (a);
\end{tikzpicture} \] 
and its graph groupoid $\mathcal{G}_{E_\infty}$. By Theorem~\ref{thm:graphhom} $H_0(\mathcal{G}_{E_\infty}) \cong \ZZ$ and $H_1(\mathcal{G}_{E_\infty}) \cong 0$. So the exact sequence in the AH~conjecture for $\mathcal{G}_{E_\infty}$ collapses to
\[\begin{tikzcd}
\mathbb{Z}_2 \arrow{r}{j} & \llbracket \mathcal{G}_{E_\infty}\rrbracket_{\text{ab}} \arrow{r} & 0.
\end{tikzcd} \]
This leaves two possibilities for the abelianization $\llbracket \mathcal{G}_{E_\infty}\rrbracket_{\text{ab}}$: 
\begin{enumerate}
\item Either $\llbracket \mathcal{G}_{E_\infty}\rrbracket_{\text{ab}}$ is trivial (in which case $\llbracket \mathcal{G}_{E_\infty} \rrbracket$ is simple),
\item or $\llbracket \mathcal{G}_{E_\infty}\rrbracket_{\text{ab}}$ is isomorphic to $\ZZ_2$ (in which case $\mathcal{G}_{E_\infty}$ has the strong AH~property).
\end{enumerate}
For $2 \leq n < \infty$ the topological full group $\llbracket \mathcal{G}_{E_n} \rrbracket$ is isomorphic to the Higman--Thompson group $V_{n,1}$~\cite{MatTFG}, and we have 
\[\llbracket \mathcal{G}_{E_n}\rrbracket_{\text{ab}} \cong (V_{n,1})_\text{ab} \cong \begin{cases}
\ZZ_2 & \quad n \text{ odd}, \\
0 & \quad n \text{ even}.
\end{cases}   \]
Although we have not been able to decide which is the case for $\llbracket \mathcal{G}_{E_\infty}\rrbracket_{\text{ab}}$, we can still deduce some structural properties of the topological full group $\llbracket \mathcal{G}_{E_\infty}\rrbracket$.

Theorem~4.16 in~\cite{MatTFG} shows not only that the commutator subgroup $\DD(\llbracket \mathcal{G}_{E_\infty} \rrbracket)$ is simple, it is also contained in any nontrivial normal subgroup of $\llbracket \mathcal{G}_{E_\infty} \rrbracket$. This means that~$\llbracket \mathcal{G}_{E_\infty} \rrbracket$ either is simple itself, or contains precisely one nontrivial normal subgroup, namely~$\DD(\llbracket \mathcal{G}_{E_\infty} \rrbracket)$ (of index $2$).  The group~$\llbracket \mathcal{G}_{E_\infty} \rrbracket$ is nonamenable~\cite{MatTFG}, but does have the Haagerup property~\cite{NO}. One can also deduce that $\llbracket \mathcal{G}_{E_\infty} \rrbracket$ is \mbox{$C^*$-simple} by the results in~\cite{BS19}. Finally, it is shown below that $\llbracket \mathcal{G}_{E_\infty} \rrbracket$ is not finitely generated.





\subsection{Simplicity and non-finite generation of topological full groups}
 
We would have liked to decide whether all graph groupoids of graphs satisfying the AH~criteria have the strong AH~property, as we know SFT-groupoids do. Matui's proof of this for SFT-groupoids in~\cite{MatTFG} relies on the construction of a finite presentation for their topological full groups. However, if a graph has infinite emitters, then the topological full group of its graph groupoid will not even be finitely generated.

\begin{proposition}
Let $E$ be a graph with at least one infinite emitter and suppose $E$ satisfies Condition~(L). Then $\llbracket \mathcal{G}_E \rrbracket$ is not finitely generated.
\end{proposition} 
\begin{proof}
Let $w \in E^0_\text{sing}$ be an infinite emitter and enumerate the edges emitted by $w$ as~${w E^1 = \{ e_1, e_2, e_3, \ldots \}}$. Suppose we are given finitely many elements~${\alpha_1, \alpha_2, \ldots, \alpha_N}$ from~$\llbracket \mathcal{G}_E \rrbracket$. According to~\cite[Proposition~9.4]{NO} we can decompose each full bisection defining these elements as 
\[U_{\alpha_j} = \left( \bigsqcup_{i=1}^{k_j} Z(\mu_{i,j}, F_{i,j}, \nu_{i,j}) \right) \sqcup (\partial E \setminus \supp(\alpha_j)). \]
Among the paths $\mu_{i,j}$ and $\nu_{i,j}$ and in the sets of forbidden edges $F_{i,j}$, only finitely many of the edges in $w E^1$ can occur. Pick an $M \in \NN$ such that $e_M, e_{M+1}, \ldots$ do not occur in any of these. Any product of the $\alpha_j$'s and their inverses will again result in an element of $\llbracket \mathcal{G}_E \rrbracket$ whose defining bisection decomposes similarly as above. And the crucial point is that none of the edges $e_M, e_{M+1}, \ldots$ will occur in its decomposition either. This means that elements such as $\pi_{\widehat{V}}$ for $V = Z(e_M, e_{M+1})$ does not belong to the subgroup generated by the elements $\alpha_1, \alpha_2, \ldots, \alpha_N$, and consequently $\llbracket \mathcal{G}_E \rrbracket$ cannot be finitely generated.
\end{proof} 

A consequence of SFT-groupoids having the strong AH~property is that their topological full groups are simple if and only if the zeroth homology group is $2$-divisible~\cite[Corollary~6.24.(3)]{MatTFG}. This is the case for e.g.~the graphs $E_n$ above when $n$ is even. For graphs with infinite emitters, however, the sitatuation is quite different. What we observed for $\mathcal{G}_{E_\infty}$ above, namely that the strong AH~property rules out the simplicity of the topological full group and vice versa, is actually a general phenomenon. This is due to $H_0(\mathcal{G}_E)$ never being $2$-divisible when $E$ has singular vertices.

\begin{proposition}
Let $E$ be a graph satisfying the AH~criteria and having at least one infinite emitter. If $\mathcal{G}_E$ has the strong AH~property, then $\llbracket \mathcal{G}_E \rrbracket$ is not  simple. 
\end{proposition}
\begin{proof}
By Theorem~\ref{thm:graphhom} $H_0(\mathcal{G}_E)$ is a finitely generated abelian group whose rank is greater than or equal to the number of singular vertices in $E$. So if $E$ has an infinite emitter, then~$H_0(\mathcal{G}_E) \otimes \ZZ_2$ is nonzero. And if $\mathcal{G}_E$ has the strong AH~property, then this forces~$\llbracket \mathcal{G}_{E}\rrbracket_{\text{ab}} \neq 0$ too. Thus  $\llbracket \mathcal{G}_{E}\rrbracket$ cannot be simple (being non-abelian).
\end{proof}

Whether or not graph groupoids of graphs with infinite emitters all have the strong AH~property can therefore be decided in the negative by finding such a groupoid whose topological full group is simple.

\subsection{Describing the abelianization of the topological full group} We first note that by Remark~\ref{rem:strongvssplit}, the abelianization~$\llbracket \mathcal{G}_{E_\infty}\rrbracket_{\text{ab}}$ is a finitely generated abelian group for any graph $E$ satisfying the AH~criteria. Let us next consider an example where both~$H_0(\mathcal{G}_E)$ and~$H_1(\mathcal{G}_E)$ are nontrivial.  

\begin{example}
Consider the graph $E$ in Figure~\ref{fig:ex}.
\begin{figure}
\[\begin{tikzpicture}[vertex/.style={circle, draw = black, fill = black, inner sep=0pt,minimum size=5pt}, implies/.style={double,double equal sign distance,-implies}]

\node at (-1,1.5) {$E$};
\node[vertex] (a) at (0,0) {};
\node[vertex] (b) at (2,0) {};
\node[vertex] (c) at (5,0) {};
\node[vertex] (d) at (4.5,2.5) {};

\path (a) edge[thick, decoration={markings, mark=at position 0.99 with {\arrow{triangle 45}}}, postaction={decorate}] (b)

(b) edge[thick, decoration={markings, mark=at position 0.99 with {\arrow{triangle 45}}}, postaction={decorate}] node[above left, near end] {$(3)$}  (d)
	edge[thick, bend left, decoration={markings, mark=at position 0.99 with {\arrow{triangle 45}}}, postaction={decorate}] node[auto] {$(3)$}  (c)
	edge[thick,loop, min distance = 15mm, looseness = 20, out = 45, in = 135, decoration={markings, mark=at position 0.99 with {\arrow{triangle 45}}}, postaction={decorate}] (b)

(c) edge[thick, implies, bend left] node[auto] {$(\infty)$}   (b)

(d) edge[thick,loop, min distance = 15mm, looseness = 20, out = 45, in = 135, decoration={markings, mark=at position 0.99 with {\arrow{triangle 45}}}, postaction={decorate}] node[above] {$(4)$} (d)
	edge[thick, decoration={markings, mark=at position 0.99 with {\arrow{triangle 45}}}, postaction={decorate}] node[auto] {$(3)$} (c)
; 
\end{tikzpicture}\]
      \caption{An infinite graph for which~$H_0(\mathcal{G}_E)$ and~$H_1(\mathcal{G}_E)$ are both nontrivial. The numbers in paranthesis indicate the number of edges.}
      \label{fig:ex}
\end{figure}
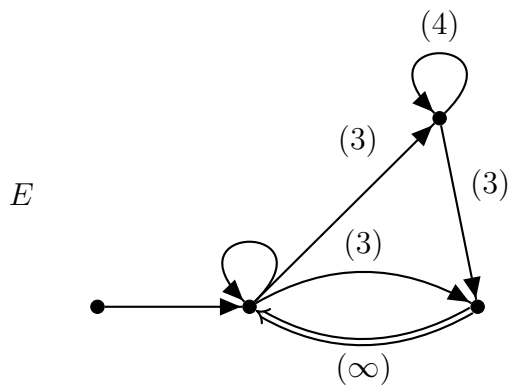
From Theorem~\ref{thm:graphhom} we find that $H_0(\mathcal{G}_E) \cong \ZZ^2 \oplus \ZZ_3$ and $H_1(\mathcal{G}_E) \cong \ZZ$. Hence the AH~exact sequence becomes
\[\begin{tikzcd}
\mathbb{Z}_2 \oplus \mathbb{Z}_2  \arrow{r}{j} & \llbracket \mathcal{G}_{E}\rrbracket_{\text{ab}} \arrow{r}{I_{\text{ab}}} & \ZZ \arrow{r} & 0.
\end{tikzcd} \]
This implies that $\llbracket \mathcal{G}_{E}\rrbracket_{\text{ab}} \cong \ZZ \oplus \im(j)$. Thus $\llbracket \mathcal{G}_{E}\rrbracket_{\text{ab}}$ is isomorphic to either $\ZZ$,  $\ZZ \oplus \mathbb{Z}_2$ or~$\ZZ \oplus \mathbb{Z}_2 \oplus \mathbb{Z}_2$.
\end{example}

The previous example generalizes to the following partial description of the abelianization~$\llbracket \mathcal{G}_{E}\rrbracket_{\text{ab}}$.

\begin{proposition}
Let $E$ be a graph satisfying the AH~criteria and let $\emptyset \neq Y \subseteq \partial E$ be clopen. Then 
\[\llbracket \mathcal{G}_{E} \vert_Y \rrbracket_{\text{ab}} \cong H_1(\mathcal{G}_E) \oplus \im(j),\]
where $H_1(\mathcal{G}_E) \cong \ZZ^M$ and $\im(j) \cong (\ZZ_2)^N$ for nonnegative integers $M, N$.  
\end{proposition} 

\begin{remark}
The integer $N$ in the preceding proposition is necessarily bounded above by the number of ``even summands'' in $H_0(\mathcal{G}_E)$, which in turn is at least $M + \vert E^0_{\text{sing}} \vert$ and at most $\vert E^0 \vert$.  In general, we may only say that~$0 \leq N \leq \vert E^0 \vert$. 
\end{remark}

\subsection{The cycle graphs} The statement in Theorem~\ref{thm:AHintro} would look cleaner if we did not have to specify that $E$ cannot be a cycle graph. However, this is necessary, as we will see shortly. Let $C_n$ denote the graph consisting of a single cycle with $n$ vertices. Observe that~${\mathcal{G}_{C_n} \cong \mathcal{R}_n \times \ZZ}$ (where $\ZZ$ is viewed as a group), which is a discrete transitive\footnote{A groupoid with only one orbit is called \emph{transitive}.} groupoid with unit space consisting of $n$ points. This is consistent with the \mbox{$C^*$-algebraic} side of things, as we have that~${C_r^*(\mathcal{G}_{C_n}) \cong C^*(C_n) \cong M_n(C(\mathbb{T}))}$ and~${C_r^*(\mathcal{R}_n \times \ZZ) \cong M_n(\mathbb{C}) \otimes C(\mathbb{T}) \cong M_n(C(\mathbb{T}))}$. Since~$\mathcal{G}_{C_n}$ is Kakutani equivalent to~$\ZZ$ and~${K_*(M_n(C(\mathbb{T}))) \cong K_*(C(\mathbb{T})) \cong (\ZZ,\ZZ)}$, Theorem~\ref{thm:graphhom} gives\footnote{We could also have deduced the homology of $\mathcal{G}_{C_n}$ from the group homology $\ZZ$, as these coincide due to their Kakutani equivalence.}
\[ H_0(\ZZ) \cong H_0(\mathcal{G}_{C_n}) \cong \ZZ
 \qquad \text{and} \qquad 
H_1(\ZZ) \cong H_1(\mathcal{G}_{C_n})  \cong \ZZ.  \]
But the unit space of $\mathcal{G}_{C_n}$ is finite, hence so is $\llbracket \mathcal{G}_{C_n} \rrbracket$ (it is isomorphic to the symmetric group~$S_n$), and then clearly the index map $I \colon \llbracket \mathcal{G}_{C_n} \rrbracket \to H_1(\mathcal{G}_{C_n})$ cannot be surjective.




\newcommand{\etalchar}[1]{$^{#1}$}

\end{document}